\documentclass{amsart}
\usepackage{amsfonts, amsthm, amssymb, amsmath, stmaryrd}
\usepackage{mathrsfs,array}
\usepackage{calligra}
\usepackage{eucal,color,times,enumerate,accents}
\usepackage[all]{xy}
\usepackage{mathtools}
\usepackage{url}
\usepackage{enumitem}
\usepackage{enumerate}
\usepackage[pdftex,bookmarksnumbered,bookmarksopen]
{hyperref}
\usepackage{lipsum}
\usepackage{adjustbox}
\usepackage{cite}
\usepackage[colorinlistoftodos]{todonotes}
\usepackage{tikz-cd}
\usetikzlibrary{decorations.pathreplacing}
\usetikzlibrary{positioning}
\usepackage{dsfont}
\usepackage{tikz}
\usetikzlibrary{matrix,arrows,decorations.pathmorphing,positioning}
\usepackage{calligra}
\usepackage[colorinlistoftodos]{todonotes}
\usepackage{xy}
\xyoption{all}
\usepackage{tikz}
\usetikzlibrary{matrix,arrows,decorations.pathmorphing}
\usepackage{leftidx}
\input xy
\xyoption{all}

\usepackage{tikz-cd}
\theoremstyle{plain}
\newtheorem{thm}{Theorem}[section]
\newtheorem{conj}{Conjecture}[section]
\newtheorem{question}{Question}[section]
\newtheorem{lemma}[thm]{Lemma}
\newtheorem*{thm*}{Theorem}
\newtheorem*{prop*}{Proposition}
\newtheorem{prop}[thm]{Proposition}

\newtheorem{cor}[thm]{Corollary}
\theoremstyle{definition}
\newtheorem{defi}[thm]{Definition}

\theoremstyle{remark}
\newtheorem{rmk}{Remark}[section]
\newtheorem{rmks}{Remarks}[section]

\numberwithin{equation}{section}
\newcommand{\FF}{\mathbf{F}}

\DeclareMathOperator{\Br}{Br}

\DeclareMathOperator{\End}{End}

\DeclareMathOperator{\Pic}{Pic}
\DeclareMathOperator{\Hom}{Hom}

\DeclareMathOperator{\Spec}{Spec}

\DeclareMathOperator{\Id}{Id}

\DeclareMathOperator{\GL}{GL}

\DeclareMathOperator{\SO}{SO}
\DeclareMathOperator{\OO}{O}

\newcommand{\id}{\operatorname{Id}}

\newcommand{\Hdg}{\operatorname{Hdg}}
\newcommand{\Aut}{\operatorname{Aut}}
\newcommand{\NS}{\operatorname{NS}}

\newcommand{\Z}{\mathbf{Z}}
\newcommand{\Q}{\mathbf{Q}}

\newcommand{\R}{\mathbb{R}}
\newcommand{\Pp}{\mathbb{P}}
\newcommand{\A}{\mathbf{A}}

\newcommand{\D}{\mathcal{D}}

\newcommand{\F}{\mathcal{F}}

\newcommand{\C}{\mathbb{C}}

\newcommand{\hZ}{\widehat{\Z}}

\newcommand{\adjunction}[4]{\xymatrix@1{#1{\ } \ar@<-0.3ex>[r]_{ {\scriptstyle #2}} & {\ } #3 \ar@<-0.3ex>[l]_{ {\scriptstyle #4}}}}

 \usepackage[foot]{amsaddr}

\makeatletter
\newcommand*{\da@rightarrow}{\mathchar"0\hexnumber@\symAMSa 4B }
\newcommand*{\da@leftarrow}{\mathchar"0\hexnumber@\symAMSa 4C }
\newcommand*{\xdashrightarrow}[2][]{%
  \mathrel{%
    \mathpalette{\da@xarrow{#1}{#2}{}\da@rightarrow{\,}{}}{}%
  }%
}
\newcommand{\xdashleftarrow}[2][]{%
  \mathrel{%
    \mathpalette{\da@xarrow{#1}{#2}\da@leftarrow{}{}{\,}}{}%
  }%
}
\newcommand*{\da@xarrow}[7]{%
  \sbox0{$\ifx#7\scriptstyle\scriptscriptstyle\else\scriptstyle\fi#5#1#6\m@th$}%
  \sbox2{$\ifx#7\scriptstyle\scriptscriptstyle\else\scriptstyle\fi#5#2#6\m@th$}%
  \sbox4{$#7\dabar@\m@th$}%
  \dimen@=\wd0 %
  \ifdim\wd2 >\dimen@
    \dimen@=\wd2 %   
  \fi
  \count@=2 %
  \def\da@bars{\dabar@\dabar@}%
  \@whiledim\count@\wd4<\dimen@\do{%
    \advance\count@\@ne
    \expandafter\def\expandafter\da@bars\expandafter{%
      \da@bars
      \dabar@ 
    }%
  }%  
  \mathrel{#3}%
  \mathrel{%   
    \mathop{\da@bars}\limits
    \ifx\\#1\\%
    \else
      _{\copy0}%
    \fi
    \ifx\\#2\\%
    \else
      ^{\copy2}%
    \fi
  }%   
  \mathrel{#4}%
}
\makeatother

\title{Rational points in the Noether-Lefschetz locus of moduli spaces of K3 surfaces}
\author{Domenico Valloni}

\begin{document}

\maketitle

\begin{abstract}
In this paper, we study maps between moduli spaces of lattice-polarized K3 surfaces induced by sublattices of prime index. We show that these maps can be used to determine if a rational point of the moduli space belongs to the Noether-Lefschetz locus. As an application, we prove that the Bombieri-Lang conjecture implies non-density statements for the rational points in the Noether-Lefschetz locus, as predicted by a conjecture of Shafarevich.
\end{abstract} 

\section{introduction}
Let $L$ be a lattice of signature $(1_+, n_-)$ that embeds primitively into the K3 lattice. Nikulin \cite{MR544937} introduced the moduli space $\mathcal{F}_L$ of $L$-quasipolarized K3 surfaces, which parametrizes isomorphism classes of pairs $(X, \iota)$, where $X$ is a complex K3 surface and $\iota \colon L \hookrightarrow \NS(X)$ is a primitive embedding satisfying an ampleness condition. The spaces $\mathcal{F}_L$ have been extensively studied in the literature, particularly for specific choices of $L$, and they are normal quasiprojective varieties of dimension $20 - \mathrm{rank}(L)$. They also have a natural model over $\Q$ and are closely related to Shimura varieties. 

In this paper, we are interested in the Noether-Lefschetz locus $\mathcal{F}_L^{nl}$ of $\mathcal{F}_L$, which consists of all the points $(X, \iota)$ for which $\iota \colon L \hookrightarrow \NS(X)$ is a proper inclusion. We remark that determining whether a given point lies in $\mathcal{F}_L^{nl}$ amounts to know something about $\NS(X)$, which is in general a difficult problem (see e.g., \cite{vanluijk2007}, \cite{zbMATH05279288} and \cite{zbMATH06322070}). To distinguish points in $\mathcal{F}_L^{nl}$ we use the negative definite lattice $\NS_{\iota}(X) \coloneqq \iota(L)^{\perp}$, and we remark that there are only finitely many subvarieties of $\mathcal{F}_L$ whose very general point is a K3 surface $(X, \iota)$ with a given $\NS_{\iota}(X)$.

Let now $K$ be a number field, and let $\mathcal{F}_L^{nl}(K)$ denote the set of $K$-rational points lying in the Noether-Lefschetz locus. According to \cite[Thm.5.2.]{PMIHES_1996__83__5_0}, the subset $\mathcal{F}_L^{nl}(K)$ is always thin. This is in fact proved by Andr\'{e} for general families of motives, with a key ingredient being a version of Hilbert irreducibility for infinite extensions proved in \cite{zbMATH00042767}. In this work, we shall use related ideas to describe the set $\mathcal{F}_L^{nl}(K)$. To state our main result, for any prime $p$ coprime to $\mathrm{discr}(L)$, we construct three branched coverings $\pi_p^{\bullet} \colon \mathcal{F}^{\bullet}_L(p) \rightarrow \mathcal{F}_L$, where $\bullet \in \{+,-,0 \}$. Each covering corresponds to an isometry class of a sublattice $T' \subset T$ such that $T/T' \cong \Z/p \Z$, and if $4 \leq \mathrm{rank}(L) \leq 18$, then each $\mathcal{F}_L^\bullet(p)$ is a moduli space of lattice polarized K3 surfaces (in general, at least one of them always is). Under some mild conditions on $L$ the spaces $\mathcal{F}_L^\bullet(p)$ also have one or two components, and $\deg(\pi_p^\bullet) \sim p^{22- \mathrm{rank}(L)}$ on each of them. Finally, one defines $\OO_\Delta(\NS_{\iota}(X)) \subset \OO(\NS_{\iota}(X))$ as the subgroup which respects effective $(-2)$-classes in $\NS_{\iota}(X)$. We note that $|\OO(\NS_{\iota}(X))| \leq 3^{21^3}$ always, and in practice, $\OO_{\Delta}(\NS_{\iota}(X)) \subset \mu_2$ for most of the lattices. We can now state our main result:

\begin{thm} \label{main thm}
Let $L$ be an hyperbolic lattice which is not $2$-elementary, and let $U \subset \mathcal{F}_L$ be the Zariski open which supports a universal family (see Proposition \ref{prop univ family}). Let $K$ be a number field and let $x \in U(K)$ be a rational point representing $(X, \iota)$. Then:

\begin{enumerate}
\item  If $x \in \mathcal{F}_L^{nl}$, there is an extension $K'/K$ with $[K' \colon K] \leq |\OO_\Delta(\NS_\iota(X))|$ and a rational point $x' \in \mathcal{F}^\bullet_{L}(p)(K')$ such that $\pi_p(x') = x$, at least for one choice $\bullet \in \{+, -,0\}$. 
\item In fact, $x$ is in the Noether-Lefschetz locus if and only if point (1) holds for $p$ sufficiently large;
\item If the cokernel of $\NS_\iota(X) \perp \iota(L) \subset \NS(X)$ is not $2$-torsion and $\OO_\Delta(\NS_\iota(X)) \subset \mu_2$ (which is automatic for instance if $\rho(X) = \mathrm{rank}(L) +1$) then we can choose $K'=K$ in point (1). 
\end{enumerate}
\end{thm}
For a closed point $x \in \mathcal{F}_{L}$ and a closed point $x' \in \mathcal{F}_{L'}$ such that $\pi^{\bullet}_p(x') = x$, one expects in general that $\mathrm{deg}(x') \sim \mathrm{deg}(x) \cdot \mathrm{deg}(\pi^{\bullet}_p)$, which can be made more precise using Hilbert irreducibility. Since the numbers $|\OO_\Delta(K)|$ are universally bounded and the degrees of $\pi_p^\bullet$ are arbitrarily large, this shows that the closed points in $\mathcal{F}_L^{nl}$ behave in a special way with respect to the coverings $\pi_p^\bullet$. We remark that these coverings already appeared in the literature; for instance, they were studied in \cite{MR1255698} when $\mathrm{rank}(L) = 1$.

Point (2) can be made uniform assuming the following conjecture of V\'{a}rilly-Alvarado:
\begin{conj}[\cite{MR3644254}] \label{conj Brauer}
Let $L$ be an even hyperbolic lattice and let $X/K$ be a K3 surface over a number field such that $\NS(\overline{X}) \cong L$. Then, $| \Br(\overline{X})^{G_K} |$ can be bounded only in terms of $[K \colon \Q]$.
\end{conj}
This and similar finiteness have also been extensively studied for instance in \cite{MR2920883}, \cite{zbMATH06149392}, \cite{Rachel}, \cite{MR3731278}, \cite{zbMATH06898536}, \cite{zbMATH07262975}, \cite{VALLONI2021107772}.
\begin{cor} \label{cor implication}
 Assume that Conjecture \ref{conj Brauer} is true for a lattice $L$, and let $N>0$ be an integer. Then, there is a number $C = C(N,L) > 0$ such that, for any number field $K/\Q$ satisfying $[K \colon \Q] \leq N$, a $K$-rational point $x \in U(K)$ belongs to the Noether-Lefschetz locus if and only if point (1) of Theorem \ref{main thm} holds for one (and hence any) $p \geq C$.
\end{cor}
In the final part of the paper, we prove our main application. We recall that the Bombieri-Lang conjecture is a generalization of the Mordell conjecture, or Faltings' theorem, which states that rational points on varieties of general type are never Zariski dense.
\begin{thm} \label{cor shafa}
Assume the Bombieri-Lang conjecture and let $K$ be a number field. 
\begin{enumerate}
    \item  If $ 1 \leq \mathrm{rank}(L) \leq 11$ then the set of all the $K$-rational points in the Noether-Lefschetz locus satisfying (3) from Theorem \ref{main thm} is not Zariski-dense.
    \item Assume that $ \mathrm{rank}(L) = 3$. Then $\mathcal{F}^{nl}(K)$ is not Zariski dense. 
\end{enumerate}
\end{thm}
 Theorem \ref{main thm} readily yields the proof of (1) once we know that the spaces $\mathcal{F}_{L}^\bullet(p)$ are of general type for sufficiently large $p$, which follows from the results in \cite{MR3275655}. On the other hand, the proof of (2) does not directly follow from our main theorem, but on other similar constructions. We remark that (2) holds without any further assumption if we have an analogue of Theorem 1.1 in \cite{MR3275655} for any rank.

The statement above is relevant only when $\mathcal{F}_L$ is not of general type. There are numerous examples where it is shown that $\mathcal{F}_L$ is either rational or unirational, see \cite{zbMATH00125557},\cite{MR1201388}, \cite{MR1397987}, \cite{MR3275655}, \cite{zbMATH07390638}, \cite{MR4297179}, as well as the references therein. Similarly, there are also examples where $\mathcal{F}_L$ is unirational and $\mathrm{rank}(L) = 3$, see for instance \cite{MEZFOR}. Nonetheless, it is expected that for all but finitely many $L$, the space $\mathcal{F}_L$ is of general type, as suggested by the results in \cite{zbMATH05199718}, \cite{MR3275655} and \cite{zbMATH07390638}.

Finally, the only unconditional result like Theorem \ref{cor shafa} is proved for any one-dimensional family in \cite[Thm.1.1]{zbMATH06218379}. After the author had written the first version of this paper, he was pointed out to the preprint \cite{CadoretPreprint} which contains similar arguments for $\Q_\ell$-local systems on any variety. In particular, in \cite{CadoretPreprint} it is explained how non-density results like Theorem \ref{cor shafa} would follow from the Bombieri-Lang conjecture together with certain geometric conjectures. In our case, thanks to Theorem \ref{main thm}, we are able to replace such conjectures with the results of \cite{MR3275655}, allowing us to prove Theorem \ref{cor shafa}.

\subsection{Acknowledgment}
This work was funded by the Europen Research Council (ERC) under the European Union's Horizon 2020 research and innovation programme - grant agreement No 948066. I express my gratitude to Klaus Hulek and Tobias Kreutz for reviewing the initial draft and providing me with important references. In addition, I also would like to thank Emiliano Ambrosi, Olivier de Gaay Fortman, Salvatore Floccari, Stefan Schreieder, Matthias Sch\"{u}tt and Alexei Skorobogatov for invaluable discussions and assistance.

\section{Moduli spaces of lattice polarized K3 surfaces} \label{section moduli}
In this chapter, we review the construction of the moduli spaces of lattice-polarized K3 surfaces. We follow \cite{zbMATH00953753} as a primary reference, along with the more recent paper \cite{https://doi.org/10.48550/arxiv.2101.12186}
\subsection{Stable orthogonal groups} A lattice $L$ is an abelian group $L \cong \Z^n$ together with a non-degenerate symmetric pairing $L \times L \rightarrow \Z$. All the lattices in this paper will be even, meaning that $(x,x) \in 2 \Z$ for any $x \in L$. The signature of $L$ is the signature of $L_\R$ and the discriminant form of an even lattice $L$ is the group $A_L = L^\vee/L$ endowed with its natural quadratic form $q \colon A_L \rightarrow \Q/ 2\Z$ (see \cite{MR525944}). One denotes by $\OO(A_L)$ the finite group of isometries of $A_L$. The isometry group of $L$ is denoted as $\OO(L)$. We denote by $\ell(A_L)$ the length of the abelian group $A_L$ (minimal number of generators) and we say that $L$ is unimodular if $\mathrm{discr}(L)= \pm 1,$ that is, if the natural inclusion $L \hookrightarrow L^\vee$ is an isomorphism. 
\begin{defi}
The stable isometry group of $L$ is $\OO^*(L) \coloneqq \ker (\OO(L) \rightarrow \OO(A_L))$. 
\end{defi}
The group $\OO^*(L)$ enjoys some well-known properties:
\begin{lemma} \label{lemma stable}
Let $L \hookrightarrow U$ be a primitive embedding of even lattices, where $U$ is unimodular. Define $\Gamma(L,U) \coloneqq \{ g \in \OO(U) \colon g_{| L^\perp } = \id \}$. Then, the natural map $\Gamma(L,U)  \rightarrow \OO(L)$ induces an isomorphism  $\Gamma(L,U)  \cong \OO^*(L).$
\end{lemma}

\begin{proof} 
Let $K$ be the orthogonal complement of $L$ in $U$. By \cite[Prop.1.6.1]{MR525944}, there is a natural isomorphism $A_L \cong -A_K$ where $-A_K$ denotes the group $A_K$ with opposite quadratic form. It follows that if $g \in \Gamma(L,U)$ then the induced isometry $g_{|L} \in \OO(L)$ must act trivially on $A_L$. On the other hand, let $\tilde{g} \in \OO^*(L)$ be a stable isometry. Due to \cite{MR525944}[Cor.1.5.2] one sees that $\tilde{g} \oplus \id_K \colon L \oplus K \xrightarrow{\sim} L \oplus K$ can always be extended to a global isometry $g \in   \Gamma(L,U)$. 
\end{proof}
In other words, an isometry of $L$ is stable if and only if it can be extended to an isometry of $U$ by requiring it to be the identity on $L^\perp$. Note that this works for every unimodular lattice $U$ which contains $L$ primitively. 
\begin{prop} \label{prop funct O^*}
Let $L \hookrightarrow N$ be any inclusion of lattices that respects the quadratic form. Then, there is a natural map $\mathrm{O}^*(L) \hookrightarrow \mathrm{O}^*(N)$ induced by extension by $1$.
\end{prop}

\begin{proof}
Let $\tilde{L}$ be the saturation of $L$ in $N$. The map $L \hookrightarrow N$ factorizes as $L \hookrightarrow \tilde{L} \hookrightarrow N$ where the first arrow exhibits $\tilde{L}$ as an overlattice of $L$ (that is, the quotient $\tilde{L}/L$ is finite) and the second arrow is a primitive embedding of lattices. It is sufficient to prove the proposition for these two kinds of maps. So assume that $N$ is an overlattice of $L$. As explained in \cite[Prop.1.4.1]{MR525944}, $N$ is determined by an isotropic subgroup $H \subset A_L$, that is, under the natural quotient map $\pi \colon L^\vee \rightarrow A_L$ one has that $\pi^{-1}(H) = N$. Clearly, if an isometry $g \in \OO(L)$ acts trivially on $A_L$ then $g$ extends uniquely to an isometry of $N$. Moreover, there is a natural isomorphism $A_N \cong H^\perp / H$, and it follows that such an extension must belong to $\OO^*(N)$. Consider now a primitive embedding $L \hookrightarrow N$, so that $N/L$ is torsion free. Let $U$ be a unimodular lattice with a primitive embedding $N \hookrightarrow U$. We obtain a chain $L \subset N \subset U$. Let $g \in \OO^*(L)$ be a stable isometry. By the lemma above, we can extend $g \oplus \id_{L^\perp}$ to an isometry $\tilde{g} \in \OO(U)$. This isometry clearly belongs to $\Gamma(N,U) \cong \OO^*(N)$, and the proposition is proved. 
\end{proof} 
\begin{rmk}
All these properties hold in general for lattices over PIDs, e.g., for $p$-adic lattices.
\end{rmk}
\subsection{Moduli spaces} 
Let $\Lambda$ denote the K3 lattice, the unique even, unimodular lattice of rank 22 and of signature $(3_+,19_-)$. An even lattice $L$ is called hyperbolic if its signature pair is of the form $(1_+, n_-)$. Two primitive embedding $f_1, f_2 \colon L \hookrightarrow \Lambda$ are equivalent if there exists $g \in \OO(\Lambda)$ such that $f_1 = g \circ f_2$. We denote by $\mathcal{E}(L)$ the set of equivalence classes of primitive embedding of $L$ into $\Lambda$. 

\begin{rmk} \label{rmk embedding}
It is known that $\mathcal{E}(L)$ is always finite, e.g., it follows from Nikulin's description of this set \cite[Prop.1.6.1]{MR525944} and the finiteness of the genus of a lattice. Moreover, for any even hyperbolic lattice $L$ of rank $\mathrm{rank}(L) \leq 20$, if $\ell(A_L) \leq 20 - \mathrm{rank}(L)$ then $\mathcal{E}(L) = \{ * \}$ due to \cite{MR525944}[Cor.1.13.3], so that every lattice of rank smaller than $10$ embeds uniquely into the K3 lattice.
\end{rmk}
For an embedding $f$ we denote by $[f]$ its class in $\mathcal{E}(L)$. For an even hyperbolic lattice $L$, write $V= \{x \in L_\R \colon (x,x)>0 \}$ and fix a primitive embedding $f \colon L \hookrightarrow \Lambda$. Consider the set $\Delta(f) \coloneqq \{ \delta \in \Lambda \setminus L^\perp \colon (\delta, \delta) = -2 \text{ and } \langle \delta, L \rangle \text{ is hyperbolic } \}$. For each $\delta \in \Delta(f)$ we obtain a hyperplane $V_\delta \subset L_\R$ given by the condition $(x, \delta) = 0$. A small cone of $L$ relative to $f$ is a connected component of the complement $V \backslash \bigcup_{\delta \in \Delta(f) } V_\delta$ (see \cite[Def.2.21]{https://doi.org/10.48550/arxiv.2101.12186}). Note that the set $\bigcup_{\delta \in \Delta(f) } V_\delta$ depends only on the class $[f] \in \mathcal{E}(L)$. 
\begin{defi}
A small cone of $L$ is a connected component $\mathcal{C}$ of $$V \setminus \bigcup_{\substack{\delta \in \Delta(e) \\ e \in \mathcal{E}(L)}} V_\delta.$$
\end{defi}

Since the set $\mathcal{E}(L)$ is finite, it follows from \cite[Prop.2.23]{https://doi.org/10.48550/arxiv.2101.12186} that the small cones are locally rational polyhedral, that is, they are given locally by the intersection of finitely many rational half spaces. 

Fix an even hyperbolic lattice $L$ that embeds primitively into the K3 lattice and a small cone $\mathcal{C}$ of $L$. The following is \cite[Def.2.16]{https://doi.org/10.48550/arxiv.2101.12186} and \cite[Prop.2.24]{https://doi.org/10.48550/arxiv.2101.12186}. 

\begin{defi} \label{defi lattice polarized}
For $e \in \mathcal{E}(L)$, an $(L,e)$-quasipolarized K3 surface is a pair $(X,\iota)$ where $X$ is a complex K3 surface and $\iota$ is a primitive embedding $\iota \colon L \hookrightarrow \NS(X)$ such that $\iota(\mathcal{C})$ contains a big and nef class, and such that the composition $L \rightarrow \NS(X) \rightarrow H^2(X, \Z(1)) \cong \Lambda$ represents $e$ (this does not depend on the last isometry). Two $(L,e)$-quasipolarized K3 surfaces $(X,\iota), (X', \iota')$ are isomorphic if there exists an isomorphism $h \colon X \xrightarrow{\sim} X'$ such that $\iota = h^* \circ \iota'$ where $h^* \colon \NS(X') \rightarrow \NS(X)$ is the natural pullback map on N\'{e}ron-Severi lattices (in particular, the isomorphism is not required to preserve the small cone.)
\end{defi}

Note that $\iota(\mathcal{C})$ contains automatically an ample class whenever $\iota(L)^\perp$ does not contain $(-2)-$classes. This is due to the definition of small cones: if $\ell \in \iota(\mathcal{C})$ is a big and nef class, then it is ample if and only if $(\ell, \delta) > 0$ for any effective $(-2)$-class $\delta \in \NS(X)$. If $\delta \notin L^\perp$, then $(\ell, \delta) \neq 0$ because $\ell \in \mathcal{C}$, so it follows that $(\delta , \ell) >0$. 

For a hyperbolic lattice $L$, its roots are given by $\Delta(L) = \{ \delta \in L \colon (\delta,\delta) = -2 \}$. The Weyl group $W(L) \subset \OO(L)$ is the group generated by the reflections in elements of $\Delta(L)$. Note that $W(L) \subset \OO^*(L)$ always, because any such reflection can be extended (as a reflection) on any even unimodular lattice $U$ which contains $L$ primitively. A Weyl chamber is a connected component of $V \setminus \bigcup_{\delta \in \Delta(L)} V_{\delta}$. Clearly, small cones are contained in Weyl chambers. It is a classical fact that $W(L)$ acts freely on $V$ with a Weyl chamber $\mathcal{K}$ as a fundamental domain, so that the set of Weyl chambers forms a torsor under $W(L)$. If $X$ is a projective K3 surface, then its K\"{a}hler cone $\mathcal{K}_X$ is a Weyl chamber of $\NS(X)$.

 The construction of the moduli space of $(L,e)$-quasipolarized K3 surfaces then follows Dolgachev's paper. Let $\mathcal{M}$ be the fine moduli space of marked K3 surfaces. This is a $20$-dimensional complex manifold which is not Hausdorff, and which parametrizes K3 surfaces $X$ together with a marking $\phi \colon H^2(X, \Z(1)) \xrightarrow{\sim} \Lambda$, up to the natural notion of isomorphism. Over $\mathcal{M}$ one has a universal family $(\mathcal{X}, \phi)$, and the period map $$p \colon \mathcal{M} \rightarrow \mathcal{D} \coloneqq \{ \sigma \in \Pp(\Lambda_\C) \colon (\sigma, \sigma) = 0, \,\, (\sigma, \overline{\sigma}) > 0 \}$$
takes a point $m \in \mathcal{M}$ to the line generated by $\phi_m(H^{2,0}(\mathcal{X}_m))$. The open subset $\mathcal{D}$ is the period domain of $\Lambda$ and for any $x \in \mathcal{D}$ seen as a line in $\Lambda_\C$, let $W_x \subset \OO(\Lambda)$ be its Weyl group, that is, the group generated by reflection in elements $\lambda \in \Lambda \cap x^\perp$ such that $ (\lambda,\lambda) = -2$. The period map $p$ is surjective \cite{MR592693} and its fibers $p^{-1} \{x\}$ are naturally torsors under $W_x$, where $g \in W_x$ acts by $(X, \phi) \mapsto (X, g \circ \phi)$. It follows that the fibers are in one-to-one correspondence to Weyl chambers in $\NS(X)$ whenever $X$ is projective. 

One then defines marked $(L,e)$-quasipolarized K3 surfaces. Fix an embedding $ L \hookrightarrow \Lambda$ which represents $e$ and a small cone $\mathcal{C}$ on $L$. 
\begin{defi}
A marked $(L,e)$-quasipolarized K3 surface is a K3 surface $X$ together with a marking $\phi \colon \mathrm{H}^2(X, \Z(1)) \xrightarrow{\sim} \Lambda$ such that $L \subset \phi(\NS(X))$ and such that the induced map $L \rightarrow \NS(X)$ is quasipolarized (the image of $\mathcal{C}$ contains a big and nef divisor). Two marked $(L,e)$-quasipolarized K3 surfaces $(X,\phi), (X', \phi')$ are isomorphic if there exists an isomorphism, $h \colon X \xrightarrow{\sim} X'$ such that $\iota = h^* \circ \iota'$ where $h^* \colon \mathrm{H}^2(X', \Z(1))  \rightarrow \mathrm{H}^2(X, \Z(1)) $ is the natural pullback map in cohomology. 
\end{defi}
To any marked $(L,e)$-quasipolarized K3 surface $(X, \phi)$ one can associate a $(L,e)$-quasipolarized K3 surface by putting $\iota_\phi =\phi^{-1}_{|L} \colon L \hookrightarrow \NS(X)$. Consider now the period domain $$\mathcal{D}_{L^\perp} \coloneqq \{ \sigma \in \Pp(L^\perp_\C) \colon (\sigma, \sigma) = 0 \text{ and } (\sigma, \overline{\sigma}) >0 \}. $$
There is a natural embedding $\mathcal{D}_{L^\perp} \subset \D$, and we put $\mathcal{M}_L = p^{-1}\{\mathcal{D}_{L^\perp} \}$; this set consists of all the marked K3 surfaces $(X, \phi)$ such that $L \subset \phi(\NS(X))$. Restricting $\mathcal{M}_L$ further $$\mathcal{M}^{qp}_L \coloneqq \{ (X, \phi) \in \mathcal{M}_L \colon (X, \phi^{-1}_{|L}) \text{ is quasipolarized}  \}$$ yields the fine moduli space of marked $(L,e)$-quasipolarized K3 surfaces. For a point $x \in \mathcal{D}_{L^\perp}$ denote by $W_x(L^\perp) \subset \OO(\Lambda)$ the group generated by reflections in vectors $\delta \in x^\perp \cap L^\perp$ such that $(\delta, \delta) = -2$. Note that $x^\perp \cap L^\perp$ is a definite lattice (it corresponds to the orthogonal of $L$ inside $\NS(X)$) so that $W_x(L^\perp)$ is finite. 
\begin{thm} \label{thm moduli marked}
The restriction of the period map to $\mathcal{M}^{qp}_L$ is surjective, and the fibers $p^{-1}\{x \}$ are naturally torsors under the finite groups $W_x(L^\perp)$. 
\end{thm}

\begin{proof}
 In fact, let $x \in \mathcal{D}_{L^\perp}$ be any period and pick a preimage $(X, \phi)$. Since we have that $ W(\NS(X)) \subset \OO^*(\NS(X))$ we get a natural embedding $W(\NS(X)) \subset \OO(\mathrm{H}^2(X, \Z))$ so that we let $ W(\NS(X))$ act on the fibers of $\mathcal{M}_L \rightarrow \mathcal{D}_{L^\perp}$. Then we simply choose $w \in W(L)$ so that $\phi( w (\mathcal{K}_X))$ contains the fixed small cone $\mathcal{C}$ so that $(X, \phi \circ w) \in \mathcal{M}_L^{qp}$. This proves surjectivity.  
 
For the other statement, let $(X, \phi)$ and $(X, \psi)$ be two different markings that yield the same $(L,e)$-quasipolarized K3 surface. Consider the Hodge isometry $\phi^{-1} \circ \psi$ of $\mathrm{H}^2(X, \Z(1))$, which is the identity on $L$. If this map is not induced by an isomorphism, then there must exists an effective $(-2)$-class $\delta \in \NS(X)$ such that $\phi^{-1} \circ \psi( \delta)$ is not effective. But then $\delta \in \NS(X)$ must belong to $L^\perp$ by the definition of small cones, i.e., that there is $w \in W(L^\perp)$ such that $(X, \psi \circ w)$ and  $(X, \phi)$ are isomorphic. It is not difficult to show that if $w \neq w'$ then $(X, \psi \circ w) \neq (X, \phi \circ w')$. This concludes the proof.
\end{proof}

\begin{rmk}
If in Definition \ref{defi lattice polarized} we replace small cones by Weyl chamber as in \cite{zbMATH00953753} the same conclusion would not hold, see \cite[Rmk. 2.30]{https://doi.org/10.48550/arxiv.2101.12186}. 
\end{rmk}
In order to construct the space $\mathcal{M}_{L,e} \subset \mathcal{M}_L$ we had to make the choice of a small cone and of an embedding $f \colon L \hookrightarrow \Lambda$ representing $e$. As explained in \cite{https://doi.org/10.48550/arxiv.2101.12186}, for any other choice, one obtains a (non-canonically) isomorphic fine moduli space of marked quasipolarized K3 surfaces. As we shall see, this ambiguity disappears when we eliminate the marking. Consider the stable group $\Gamma = \{ g \in \OO(\Lambda) \colon g_{|L} = \mathrm{Id}_L \}.$ Note that $\Gamma \cong \OO^*(L^\perp)$ acts both on $\mathcal{M}^{qp}_{L,e}$ and on $\mathcal{D}_{L^\perp}$. For the following, see \cite[Cor.2.27]{https://doi.org/10.48550/arxiv.2101.12186} and the discussion after \cite[Prop.3.3]{zbMATH00953753}. 
\begin{prop}  \label{prop moduli}
The period map induces an isomorphism $$\mathcal{F}_{L,e} \coloneqq \mathcal{M}^{qp}_{L,e}/ \Gamma \cong \mathcal{D}_{L^\perp}/ \Gamma.$$
The points of $\mathcal{F}_{L,e}$ are in one-to-one correspondence to isomorphism classes of $(L,e)$-quasipolarized K3 surfaces. 
\end{prop}
It follows that for any other choice of the embedding $L \hookrightarrow \Lambda$ representing $e$ and for any other choice of a small cone, the resulting moduli spaces are naturally isomorphic. In fact, if instead of $j \colon L \subset \Lambda$ we choose $g \circ j$ for some $g \in \OO(\Lambda)$, then the moduli spaces $\mathcal{M}_{L}^{qp}$ and $\mathcal{M}^{qp}_{g(L)}$ are clearly isomorphic via the map on period domains induced by $g$. If $\mathcal{C}'$ is another small cone on $L$, then the fine moduli spaces $\mathcal{M}_{L, \mathcal{C}}^{qp}$ and $\mathcal{M}_{L, \mathcal{C}'}^{qp}$ are identified after the quotient by $\Gamma$. 

The manifold $\mathcal{D}_{L^\perp}^{\pm}$ consists of at most two connected components, and it is a hermitian domain of type IV. The group $\OO^*(L^\perp) \cong \Gamma$ is an arithmetic subgroup, and from Baily-Borel it follows that $\mathcal{F}_{L,e}$ is a quasi-projective normal variety with at most two connected components. We remark that $\OO^*(L^\perp)$ is not neat in general, as it may contain torsion, such as reflections. 

The definition of $L$-quasipolarized K3 surfaces can also be extended over families. Let $S/ \C$ be a noetherian scheme, and let $\pi \colon X \rightarrow S$ be a smooth projective family of K3 surfaces. Let $\Pic(X/S)$ be the relative Picard sheaf of the family (see \cite{zbMATH05518937}[Chapter 3]) and let $L_S$ denote the locally constant sheaf on $S$ induced by the lattice $L$.
\begin{defi}
We say that $X/S$ is $(L,e)$-quasipolarized if there an embedding of sheaves $\iota \colon L_S \hookrightarrow \Pic(X/S)$ which is primitive and $(L,e)$-quasipolarized for every geometric fiber of $S$. 
\end{defi}
To show that this is well defined, denote by $j_s \colon L \hookrightarrow \Pic(\mathcal{X}_s)$ the induced primitive embedding for any geometric point $s \in S(\C)$. Since $S$ is connected the class $[L \hookrightarrow \Pic(\mathcal{X}_s) \hookrightarrow H^2(\mathcal{X}_s, \Z(1)) ] \in \mathcal{E}(L)$ does not depend on the chosen $s$. Note also that a class $\ell \in \mathcal{C}$ that is big and nef on one fiber, it is big and nef on every fiber. Using the argument of \cite[Rmk. 3.4]{zbMATH00953753}, one constructs a unique map $S \rightarrow \mathcal{F}_{L,e}$ such that $s \mapsto [(\mathcal{X}_s, \iota_s)]$, showing that $\mathcal{F}_{L,e}$ coarsely represents the associated moduli functor. Finally, we get rid of the embedding $e$:
\begin{defi}
A $L$-quasipolarized K3 surface is a pair $(X, \iota)$ where $X$ is a complex K3 surface and $\iota$ is a quasipolarized embedding $L \hookrightarrow \NS(X)$. 
\end{defi}
The space $\mathcal{F}_{L} \coloneqq \bigsqcup_{e \in \mathcal{E}(L)} \mathcal{F}_{L,e}$ is then the coarse moduli space of $L$-quasipolarized K3 surfaces. The following also is well-known, but we prove it nevertheless due to a lack of references
\begin{prop} \label{thm defined Q}
The space $\mathcal{F}_{L}$ has a natural model over $\Q$.
\end{prop}
\begin{proof}
The correct way to prove this would be to follow Rizov's steps \cite{2005math......8018R}, which would also yield a model of the moduli space over some open subset of $\Spec(\Z)$, but the argument would be too long for this paper, so we follow some different but standard techniques. For a prime number $p>2$ coprime to $\mathrm{discr}(\NS(X))$ consider the group $$\Gamma_p = \ker ( \Gamma \rightarrow \OO(T \otimes \FF_p) )$$ where $\Gamma$ is as in Proposition \ref{prop moduli} and $T = L^\perp$. Note that $\Gamma_p$ is neat, so that $\mathcal{F}_p = \mathcal{D}_T/ \Gamma_p$ is smooth and quasi-projective. It follows from standard arguments that one has a family of $(L,e)$-quasipolarized K3 surfaces over $\mathcal{F}_p$, endowed with an isometry $\ell \colon \iota(L)^\perp \xrightarrow{\sim} T \otimes \FF_p$. Rigidity arguments (see \cite{MR791585} and \cite{MR3656469}) ensure that the variety $ \mathcal{F}_p$ can be defined over $\overline{\Q}$. Then, \cite[Thm.3.21]{MR2192012} says that $\Aut(\mathcal{F}_p)$ is finite, so that every model over $\overline{\Q}$ is $\overline{\Q}$-isomorphic. This also means that there is a canonically defined subset in $ \mathcal{F}_p (\C)$ of algebraic points. We denote this $\overline{\Q}$-model by $\mathcal{F}_p$, so that the original one over $\C$ becomes $\mathcal{F}_{p,\C}$. Let $(X, \iota)$ be a complex $(L,e)$-quasipolarized K3 surface and let $x \in \mathcal{F}_p(\C)$ be any point that maps to the period point of $(X, \iota)$. We want to show that $X$ is defined over $\overline{\Q}$ if and only if $x$ maps to $ \mathcal{F}_p (\overline{\Q})$. One direction is obvious due to the existence of the family over $\mathcal{F}_{p,\C}$, for the other direction we use \cite[Lem.4.3]{MR3951650} which says that a K3 surface is defined over $\overline{\Q}$ if and only if the set $\{X^\sigma \colon \sigma \in \Aut(\C / \overline{\Q}) \}  = \{X\}$ (otherwise, it is uncountably infinite). Since the same characterization holds for the algebraic points in $\F_p(\C)$, the claim is proved. The group $\Gamma_p$ is normal in $\Gamma$, so that $\Gamma/ \Gamma_p$ acts on $\mathcal{F}_p$. Taking the quotient yields a normal quasiprojective variety which is a $\overline{\Q}$-form of $\mathcal{F}_{L,e}$. We denote this variety by $\mathcal{F}_e$. Then, $\mathcal{F}_e(\overline{\Q})$ parametrizes isomorphism classes of $(L,e)$-quasipolarized K3 surfaces that can be defined over $\overline{\Q} \subset \C$ (note that we always work with a fixed $\C$). In order to show that this is also a coarse moduli space, one notes that for any family of $(L,e)$-quasipolarized K3 surfaces $\mathcal{X}/S$ where $S / \overline{\Q}$ is a noetherian base, one has an induced map $S_\C \rightarrow \mathcal{F}_{e,\C}$ and this is defined over $\overline{\Q}$ because $S(\overline{\Q}) \subset S(\C)$ must be mapped to $\mathcal{F}_e(\overline{\Q})$. Define $\mathcal{F}_L = \bigsqcup_e \mathcal{F}_{L,e}$. In order to conclude the proof we need to show that $\mathcal{F}_L$ descends to $\Q$. There is a natural $G_\Q$-action on $\mathcal{F}_L(\overline{\Q})$ defined as follows: for any $\sigma \in G_\Q$ and any $(X, \iota) \in \mathcal{F}_L(\overline{\Q})$ let $(X, \iota)^\sigma = (X^\sigma, \sigma_* \circ \iota)$ where $\sigma_* \colon \NS(X) \rightarrow \NS(X^\sigma)$ is the natural induced map. Note that $\sigma_*$ behaves as if it was induced by an isomorphism on N\'{e}ron-Severi groups because it sends ample classes to ample class, which implies that also  $(X, \iota)^\sigma$ is quasipolarized. At the level of points of $\mathcal{F}_L(\overline{\Q})$, we denote this by $x \mapsto \sigma(x)$. Let $\sigma \in G_\Q$ and consider the space $\mathcal{F}^\sigma_L$ together with the map $\sigma_{\mathcal{F}} \colon \mathcal{F}_L \rightarrow  \mathcal{F}^\sigma_L$. For any family $\mathcal{X}/S$ let $p_S \colon S \rightarrow \mathcal{F}_e$ be the period map and consider the composition $$S \xrightarrow{\sigma^{-1}_S} S^{\sigma^{-1}} \xrightarrow{p_{S^{(\sigma^{-1})}}} \mathcal{F}_L \xrightarrow{\sigma_\mathcal{F}} \mathcal{F}^\sigma_L.$$ This induced map is $\overline{\Q}-$linear and exhibits $\mathcal{F}^\sigma_L$ as a coarse moduli space of $L$-quasipolarized K3 surfaces. Since coarse moduli spaces are unique up to unique isomorphism, we obtain a unique isomorphism $\pi^\sigma \colon \mathcal{F}^{\sigma}_L \xrightarrow{\sim} \mathcal{F}_L$. We leave to the reader to check that this defines a Galois descent datum in the sense of \cite[Prop.4.4.4.]{poonen2017rational}. By \cite[Cor.4.4.6.]{poonen2017rational} and \cite[Rmk. 4.4.8.]{poonen2017rational} the space $\mathcal{F}_L$ is defined over $\Q$. We call this model again by $\mathcal{F}_L$. Let $(X, \iota)$ be a $L$-quasipolarized K3 surface over $\overline{\Q}$ and let $x$ be its period point, $x=p(X, \iota)$. Then, its period point for $\mathcal{F}^\sigma$ is given by applying $\sigma_\mathcal{F}$ to $\sigma^{-1}(x)$, where the latter action is the one described in the paragraph above. The composition $$\mathcal{F}_L(\overline{\Q}) \xrightarrow{\sigma_\mathcal{F}} \mathcal{F}^\sigma_L(\overline{\Q})  \xrightarrow{\pi^\sigma } \mathcal{F}_L(\overline{\Q})$$ describes by construction the Galois action on $\mathcal{F}_L(\overline{\Q})$, and it gives back the natural $G_\Q$-action on $\mathcal{F}_L(\overline{\Q})$.    
\end{proof}
The proposition also shows that for any number field $K$ we have $$\mathcal{F}_L(K) = \{ (X, \iota) \in \mathcal{F}_L(\overline{\Q}) \colon (X^\sigma, \iota^\sigma) \cong (X, \iota) \text{ for every } \sigma \in G_K \}.$$ 

\begin{rmk}
Although $\mathcal{F}_L$ has usually one or two connected components, there are still instances in which the set $\mathcal{E}(L)$ is not trivial (see subsection \ref{sec conn comp} for a discussion on connected components). In this case, it would be interesting to know how $G_\Q$ acts on such components. For a concrete example, take any K3 surface $X/ \overline{\Q}$ and $\sigma \in G_\Q$ such that $T(X)$ and $T(X^\sigma)$ are not isometric (once we base change both surfaces to $\C$). In this case, $\NS(X) \cong \NS(X^\sigma)$ but the two embeddings $\NS(X) \hookrightarrow \Lambda$ and  $\NS(X^\sigma) \hookrightarrow \Lambda$ cannot be isomorphic, since their orthogonal complements are not isometric. We can use the result of Sch\"{u}tt \cite{MR2602669} to provide many such examples of Picard rank $20$. On the other hand, this is not completely satisfactory because the transcendental lattices involved are positive definite. One may still wonder whether $G_\Q$ acts always transitively on $\mathcal{E}(L)$ when $L$ is not definite.
\end{rmk}
\subsection{Hodge theoretic interpretation} \label{rmk independence embedding}
We begin with the following definition. Let $T$ be an even lattice of signature $(2_+, n_-)$.

\begin{defi} \label{defi T polarized}
A $T$-polarized K3 structure is a tuple $(T', \gamma')$ where $T'$ is an even lattice in the same genus of $T$ endowed with a Hodge structure of weight zero of K3 type, for which the quadratic form of $T'_\Q$ induces a polarization, and $\gamma' \colon A_{T} \xrightarrow{\sim} A_{T'}$ is an isomorphism of finite quadratic forms. Another tuple $(T'', \gamma'')$ is isomorphic to the first one, if there is an integral Hodge isometry $f \colon T' \rightarrow T''$ such that, if $\bar{f} \colon A_{T'} \rightarrow A_{T''}$ denotes the natural induced isometry on the discriminant forms, then $\gamma' = \gamma'' \circ \bar{f}$. 
\end{defi}

 Let $\mathcal{K}^*(T)$ be the set of all isomorphism classes of $T$-polarized K3 structures. We define $\pi_0(\mathcal{K}^*(T))$ as the isomorphism classes of pairs $(T', \gamma')$ where $T'$ now has no Hodge structure, but it is only an even lattice in the same genus of $T'$ (the rest of the definition remains the same). 

\begin{lemma}
Let $L$ be an even hyperbolic lattice, $f \colon L \hookrightarrow \Lambda$ a primitive embedding, and $T$ its orthogonal complement. There is a natural isomorphism $\pi_0(\mathcal{K}^*(T)) \cong \mathcal{E}(L)$.
\end{lemma}

\begin{proof}
Due to \cite[Prop.1.6.1]{MR525944} we can identify $\mathcal{E}(L)$ with isomorphism classes of pairs $(T', \rho')$ where $\rho' \colon A_{T'} \xrightarrow{\sim} -A_L$ is an isometry and $T'$ is a lattice in the same genus as $T$. Two such pairs give isomorphic primitive embeddings if and only if there is an isometry $f \colon T' \rightarrow T''$ such that $\rho' = \rho'' \circ \bar{f}$. Using the natural isomorphism $g \colon A_T \xrightarrow{\sim} -A_L$ induced by the identification $T = L^\perp$ we can associate to any $(T', \gamma') \in \mathcal{K}_0^*(T)$ the element $(T', g \circ \gamma^{-1}) \in \mathcal{E}(L).$ This is a bijection.
\end{proof}
One has a natural surjective map $\mathcal{K}^*(T) \rightarrow \pi_0(\mathcal{K}^*(T))$. We now show that $\mathcal{F}_L(\C)$ also parametrizes $T$-polarized Hodge structures of K3 type:
\begin{prop}
Let $L$ be an even hyperbolic lattice with a primitive embedding $L \hookrightarrow \Lambda$ such that $L^\perp = T$. There is a natural identification $$\mathcal{K}^*(T) \cong \mathcal{F}_L(\C)$$ such that the following square commutes:
\begin{center}
     \begin{tikzcd}
  \mathcal{K}^*(T) \arrow[r, "\sim"] \arrow[d]
    & \mathcal{F}_L(\C) \arrow[d] \\
  \mathcal{K}_0^*(T) \arrow[r, "\sim"]
&  \mathcal{E}(L). \end{tikzcd}
\end{center}
\end{prop}
\begin{proof}
To any $L$-quasipolarized K3 surface $(X, \iota)$ we let $(\iota(L)^\perp, \gamma_X)$ be the pair where $\gamma_X$ is given by the composition $A_{\iota(L)^\perp} \cong -A_L \cong A_T.$ It is easy to see that this defines an injection $\mathcal{F}_L(\C) \hookrightarrow \mathcal{K}^*_T$. 

To go the other way around, choose any $(T', \gamma) \in \mathcal{K}^*_T$. We can glue the direct sum $T' \oplus L$ into the K3 lattice, using the isomorphism $A_{T'} \cong -A_L$ induced by $\gamma$. We then endow the K3 lattice $\Lambda$ with this Hodge structure and, by the surjectivity of the period map and the global Torelli theorem, we obtain a unique K3 surface $X$ together with an isometry $f \colon H^2(X, \Z(1)) \xrightarrow{\sim} \Lambda$. Now, choose a small cone $\mathcal{C}$ on $L$, and consider the embedding $L \hookrightarrow H^2(X, \Z(1))$ given by the composition of $L \subset L \oplus T' \subset \Lambda$ and $f^{-1}$. Pick any $w \in W(L)$ such that $\iota \coloneqq j \circ w$ is quasipolarized with respect to $\mathcal{C}$ and consider the corresponding $L$-quasipolarized K3 surface $(X, \iota)$. We need to show that $(X, \iota)$ does not depend on these choices (the K3 surface $X$ is not naturally associated to the Hodge structure, and $f$ is arbitrary). Let $(X', \iota')$ be another $L$-quasipolarized K3 surface obtained via the same procedure. Choose any Hodge isometry $f \colon H^2(X, \Z) \rightarrow H^2(X', \Z)$ such that $f \circ \iota = \iota'$. Assume that $L^\perp$ does not contain $(-2)$-classes, and let $\delta \in \NS(X)$ be any effective $(-2)$-class and $c \in \mathcal{C}$ be any element. Then $(\iota(c), \delta) >0$ and therefore $(\iota'(c), f(\delta) >0$. This shows that $f$ is induced by a unique isomorphism (by the global Torelli) of $L$-quasipolarized K3 surfaces. If $L^\perp$ contains $(-2)$-classes, then $f$ is not necessarily induced by an isomorphism. On the other hand, we can pick any element $w$ in the finite group $W(L^\perp) \subset \OO(H^2(X, \Z))$ such that $f \circ w$ is induced by a unique isomorphism. In this case, the two $L$-quasipolarized K3 surfaces are also isomorphic.
\end{proof}
Thus, in the end, one could rephrase the notion of $L$-polarized K3 surface as follows:
\begin{defi}
An $L$-polarized K3 surface is a pair $(X, \iota)$ where $X$ is a K3 surface and $\iota \colon L \hookrightarrow \NS(X)$ is a primitive embedding of lattices (without any ampleness condition). Two pairs $(X, \iota)$, $(X', \iota')$ are isomorphic if there exists an isometry of Hodge structures $f \colon H^2(X, \Z(1)) \xrightarrow{\sim} H^2(X', \Z(1))  $ such that $f \circ \iota = \iota'$.
\end{defi}
The two notions coincide but what we have said until now, and the advantage of the definition above is that it avoids all the discussion about small cones. 

On $\mathcal{K}^*_T$ there is a natural action of $\OO(A_T)$, given by $  (T', \gamma') \cdot \gamma \coloneqq (T',  \gamma' \circ \gamma)$ for any $\gamma \in \OO(A_T)$. The corresponding action on $\mathcal{F}_L$ is also defined over $\Q$. 
\begin{cor}
    The quotient of $\mathcal{M}_T/ \OO(A_T)$ parametrizes Hodge structures $T'$ where $T'$ is in the same genus as $T$. 
\end{cor}
In particular, since when $\rho(X) \geq 12$ the transcendental lattice of $X$ determines its isomorphism class, we obtain
\begin{cor}
    For any $L$ with $\mathrm{rank}(L) \geq 12$ the Picard generic points of $\mathcal{F}_L/ \OO(A_L)$ are in one-to-one correspondence to isomorphism classes of K3 surfaces $X$ with $L \cong \NS(X)$.
\end{cor}
\subsection{Double quotient interpretation} 
 Let $T$ be any even lattice of signature $(2_+, n_-)$, and let $\mathcal{D}^{\pm}_T \subset \Pp(T_\C)$ be the associated period domain. Let $\OO_{T}^*(\widehat{\Z})$ be the kernel of $\OO_T(\widehat{\Z}) \rightarrow \OO(A_L).$ Consider the following double quotient \begin{equation}
    \mathcal{M}_T \coloneqq \OO_T(\Q) \backslash \mathcal{D}_{T}^\pm \times \OO_T(\A_f) / \OO_{T}^*(\widehat{\Z}),
\end{equation}
where the action of $q \in \SO_T(\Q)$ is given by $q(\sigma, g) = (q\sigma, q g)$ and the action of $K$ is the right multiplication on the right coordinate. The product can be taken both with $\OO(\A_f)$ endowed with the adelic topology, or with the discrete topology: this does not make a difference when one goes to the quotient.
Note that $\mathcal{M}_T$ is not a Shimura variety because $\OO_T$ is not connected. We now show that $\mathcal{M}_T$ is isomorphic to $\mathcal{F}_L$ when $L$ embeds primitively into the K3 lattice with orthogonal complement isomorphic to $T$, by showing that also $\mathcal{M}_T$ parametrizes $T$-polarized Hodge structures of K3-type.
\begin{lemma}
There is a natural isomorphism $\pi_0(\mathcal{K}^*(T)) \cong \OO_T(\Q) \backslash \OO_T(\A_f) / \OO_{T}^*(\widehat{\Z})$. 
\end{lemma}

\begin{proof}
Consider $g \in \OO_T(\A_f)$. Then, $T^g \coloneqq g(T \otimes \widehat{\Z}) \cap T_\Q$ is the unique lattice that satisfies $T^g \otimes \widehat{\Z} = g(T \otimes \widehat{\Z})$ where the equality is taken in $T \otimes\A_f$. The map $g$ also induces a natural isomorphism $\gamma_g \colon A_T \rightarrow A_{T^g}$, so we obtain a map $\OO(\A_f) \rightarrow \pi_0(\mathcal{K}^*(T))$ sending $g$ to $(T^g, \gamma_g)$. The surjectivity of this map is due to the fact that $\OO(\widehat{\Z}) \rightarrow \OO(A_T)$ is surjective \cite{MR525944}[Cor.1.9.6] and it is straightforward to check that it induces a bijection between the double coset and $\pi_0(\mathcal{K}^*(T))$. 
\end{proof}
\begin{prop} \label{prop M and K}
For any even lattice $T$ of signature type $(2_+, n_-)$ there is a natural identification $$\mathcal{M}_T(\C) \cong \mathcal{K}^*(T)$$ and a natural isomorphism $$\mathcal{M}_T(\C) \cong \bigsqcup_{(T',\gamma') \in \pi_0(\mathcal{K}^*(T))} \mathcal{D}_{T'}^{\pm}/\OO^*(T').$$
\end{prop}
\begin{proof}
We put $V = T_\Q$. Let $(\sigma, g) \in \mathcal{D}_{T}^{\pm}\times \OO(\A_f)$ and denote by $[\sigma, g] \in \mathcal{M}_T(\C)$ its equivalence class. Denote by $V_\sigma$ the polarized Hodge structure induced by $\sigma \in \mathcal{D}_T \subset \Pp(T_\C)$, that is, $V^{1,-1} = \sigma$, $V^{-1,1} = \bar{\sigma}$ and $V^{0,0} \cap V_\R =  ((\sigma \oplus \bar{\sigma}) \cap T_\R)^\perp$. If there is a Hodge isometry $f \colon V_\sigma \rightarrow V_\tau$ then $f$ is an element of $\OO(\Q)$ such that $f(\sigma) = \tau$. Consider now the unique lattice $T^g \subset V$ such that $T^g \otimes \widehat{\Z} = g(T \otimes \widehat{\Z})$ with equality occurring in $T \otimes \A_f$. Denote by $T^g_\sigma$ the integral Hodge structure induced by the inclusion $T^g \subset V_\sigma$. Note that $g$ induces an isomorphism $\gamma_g \colon A_{T} \rightarrow A_{T^g}$, so we obtain an element $(T^g_\sigma, \gamma_g) \in \mathcal{K}^*(T)$. Clearly $T^{gk} = T^{k}$ for any $k \in \OO_T(\widehat{\Z})$ and $\gamma_{gk} = \gamma_g$ for any $k \in \OO_T^*(\widehat{\Z}).$ This shows that the map is well defined. To construct the inverse, one picks an element $(T', \gamma')$ and one chooses any isometry $T'_\Q \cong V$. The image of $T' \subset V$ is a lattice in the same genus of $T$, because $T'$ has the same signature and the same discriminant form of $T$. We can find both a Hodge structure $\sigma$ and an element $g \in \OO(\A_f)$ such that $T' \cong T^g_{\sigma}$ as integral polarized Hodge structures. We only need to show that we can choose $g$ such that $(T^g_\sigma, \gamma_g) \cong (T', \gamma')$. This follows from the fact that the map $\OO(\widehat{\Z}) \rightarrow \OO(A_L)$ is surjective \cite{MR525944}[Cor.1.9.6.] and from the fact that for any $h \in \OO(\widehat{\Z})$ one has $T^{gh} = T^{g}$. These two constructions are one the inverse of the other. 

We now prove the last part of the proposition. The discussion above shows how to any $[\sigma, g] \in \mathcal{M}_T(\C)$ we can associate an element $(T^{g}, \gamma_g) \in \pi_0(\mathcal{K}^*(T)).$ Assume that for another $[\sigma', g']$ we have that $(T^{g'}, \gamma_{g'}) \cong (T^g, \gamma_g)$. Note that we are now considering only lattices, without Hodge structures. This means that we can find $q \in \OO(\Q)$ and $k \in \OO^*(\widehat{\Z})$ such that  $qg' k = g$; hence, the fiber of the map $\mathcal{M}_T(\C) \rightarrow \pi_0(\mathcal{K}^*(T))$ are represented by $(\sigma, g)$ for a fixed $g \in \OO(\A_f)$ and $\sigma$ varying in $\mathcal{D}^{\pm}_T$. Now, two elements $(\sigma, g)$ and $(\sigma', g)$ map to the same one in $\mathcal{M}_T$ if and only if there is $q \in \OO(\Q)$ and $k \in \OO^*(\widehat{\Z})$ such that $\sigma' = q \sigma$ and $qgk = g$, i.e., if and only there is $q \in \OO_T(\Q)$ such that $g^{-1}qg \in \OO_{T}^*(\widehat{\Z})$ and $\sigma' = q\sigma$. We can view $g$ as an isometry between $T \otimes \widehat{\Z}$ and $T^g \otimes \widehat{\Z}$, so the equation above says that $q$ must act integrally on $T^g$, i.e., $q \in \OO_{T^g}(\Z)$. Moreover, since $k$ acts trivially on $A_T$ one can deduce from the same equation that $q \in \OO_{T^g}^*(\Z)$. Now, note that for any $q \in \OO_{T^g}^*(\Z)$ there is $k \in \OO_T^*(\widehat{\Z})$ such that $qgk =g$, simply by putting $k = gqg^{-1}$. So the fiber over $(T^{g}, \gamma_g) \in \pi_0(\mathcal{K}^*(T)) $ is naturally isomorphic to $ \mathcal{D}_{T}^{\pm}/\OO^*(T^g)$.
\end{proof}
Putting together the previous points, we have arrived at the following. Let $L$ be an even hyperbolic lattice, and fix a primitive embedding $L \subset \Lambda$ with orthogonal complement $T$. Then, the choice of such embedding induces a natural isomorphism $\mathcal{E}(L) \cong \pi_0(\mathcal{K}^*(T))$ and $\mathcal{F}_L \cong \mathcal{M}_T$ such that the following diagram commutes 
\begin{equation} \label{final diagram comparison}
     \begin{tikzcd}
  \mathcal{M}_T \arrow[r, "\sim"] \arrow[d]
    & \mathcal{F}_L(\C) \arrow[d] \\
  \mathcal{K}_0^*(T) \arrow[r, "\sim"]
&  \mathcal{E}(L). \end{tikzcd}
\end{equation}
commutes. 

\subsubsection{Connected components} \label{sec conn comp}
We comment more on the connected components of $\mathcal{F}_L$. There are two reasons why $ \pi_0(\mathcal{K}^*_T)$ may be non-trivial. The first reason is that $T$ is not unique in its genus, and the second is that $\OO(T) \rightarrow \OO(A_T)$ is not surjective. Denote by $\overline{\OO}(T) \subset \OO(A_T)$ its image. 
\begin{lemma}
    Assume that $T$ is unique in its genus. Then  $|\pi_0(\mathcal{K}^*_T)| = [ \OO(A_T) \colon \overline{\OO}(T)]$. 
\end{lemma}

\begin{proof}
In fact, any pair $(T', \gamma')$ we can find an isometry $T \xrightarrow{\sim} T'$. This means that each element is represented by one of the form $(T, \gamma)$ with $\gamma \in \OO(A_T)$, and then the proof is obvious. 
\end{proof}
\begin{cor}
    Assume that $T$ is unique in its genus. Then the connected components of $\mathcal{F}_L$ are permuted by the natural action of $\OO(A_L)$ on  $\mathcal{F}_L$. 
\end{cor}
We remark that most of the indefinite lattices have trivial genus (that is, consisting of one isometry class). As originally done by Eichler, one decomposes the genus of $T$ as a disjoint union of spinor genera. Then, Eichler showed that whenever $\mathrm{rank}(T) \geq 4$ and $T$ is indefinite, there is only one isometry class in each spinor genus. Later, Kneser \cite{zbMATH03119392} showed that the number of different spinor genera in the genus of $T$ is always a power of $2$, and most of the time, this power is $1$. For instance, for any prime $p$, consider the $p$-adic lattice $T_p = T \otimes \Z_p$, and write $T_p = U_0 \perp U_1(p) \perp \cdots \perp U_k(p^k)$ where each $U_i$ is unimodular and $U_i(p^i)$ is the lattice $U_i$ scaled by $p^i$. Let $A_{T_p} = A_T \otimes \Z_p$ be the $p$-part of the discriminant form of $T$. Then $$A_{T_p} = \bigoplus_{i} (\Z/ p^i \Z)^{\mathrm{rank}(U_i)}.$$
\begin{prop}(\cite{zbMATH03613145}[Cor.of Lem. 3.7.])
    Assume that $\mathrm{discr}(T)$ is odd and that the genus of $T$ contains more than one spinor genus. Then $\mathrm{rank}(U_i) \in  \{1,0 \}$ for each $i$ for at least one prime $p$ which divides $\mathrm{discr}(T)$.
\end{prop}
See the next corollary in \textit{loc.cit} for the case $p=2$. Finally, the following follows from Nikulin results:
 \begin{lemma} \label{lemma conn comp M_T}
     Assume that a lattice $T$ of signature $(2_+,n_-)$ satisfies $\ell(A_T) \leq \mathrm{rank}(T) -2$. Then $\mathcal{M}_T$ has at most two connected components. If moreover $\ell(A_T) \leq \mathrm{rank}(T) -3$ then it is irreducible. 
 \end{lemma}
\begin{proof}
The first statement is due to \cite[Thm.1.14.2]{MR525944}. Regarding the second, since $\ell(A_T) \leq \mathrm{rank}(T) -3$ there is an orthogonal decomposition $T= U \oplus T'$ where $U$ is the hyperbolic plane \cite{MR525944}[Cor.1.13.5]. Let $i \colon T \rightarrow T$ be the map induced by $-\Id$ on $U$ and $\Id$ on $T'$. Then $i$ clearly exchanges $\mathcal{D}^+$ with $\mathcal{D}^{-}$.
\end{proof}

\subsection{Fields of moduli and rigidity}
Let $(X, \iota)$ be a $L$-quasipolarized K3 surface over $\overline{\Q}$. The field generated by the point in $\mathcal{F}_L(\overline{\Q})$ representing $(X, \iota)$ is called the field of moduli of $(X, \iota)$. This can also be defined as the fixed field of $$\{ \sigma \in G_\Q \colon (X, \iota) \cong (X^\sigma, \iota^\sigma) \} \subset G_\Q.$$
In general, the field of moduli is smaller than (any) field of definition. For a lattice polarized K3 surface, most of the times the two fields coincide due to the scarcity of automorphisms. In general, an argument of Galois descent shows that the two fields coincide whenever the parametrized object has no non-trivial automorphisms. Denote by $\mu(X)$ the set of Hodge isometries of $T(X_\C)$. Note that this corresponds to the roots of unity in the field $E_X = \End_{\Hdg}(T(X_\C)_\Q)$ and, in particular, $\mu(X) = \{ \pm 1 \}$ whenever $E_X$ is totally real. 
\begin{prop} \label{prop fields of moduli}
Let $(X, \iota)$ be a $L$-quasipolarized K3 surface, and assume that $X$ can be defined over $\overline{\Q}$. Assume that $\iota$ is an isomorphism between $L$ and $\NS(X)$. If the natural map $\mu(X) \rightarrow \OO(A_X)$ is injective, then $(X, \iota)$ can be defined over its field of moduli. 
\end{prop}
\begin{proof}
Let $(X, \iota)$ and $(Y, \iota)$ be $(L,e)$-quasipolarized K3 surfaces, and assume that both $\iota$ and $\iota'$ are isomorphisms and that $(X,\iota) \cong (X', \iota').$ We begin by showing that there is a unique isomorphism between $(X,\iota)$ and $(X', \iota').$ Pick two isomorphisms $f,f'$ and consider $g = f' \circ \pi^{-1} \in \Aut(X)$. Then, $g$ is an automorphism of $X$ which acts trivially on $\NS(X)$ and therefore on $A_X$. Its action on $T(X)$ is given by a integral Hodge isometry, that is, by an element of $\mu \in \mu(X)$. Since the action of $\mu$ on $A_X$ must be the identity due to the fact that $T(X)^\vee/T(X)$ and $\NS(X)^\vee/\NS(X)$ are naturally isomorphic, we conclude thanks to the assumption that $\mu(X) \rightarrow \OO(A_X)$ is injective. Thus, for any $\sigma \in G_\Q$ such that  $(X, \iota) \cong (X^\sigma, \iota^\sigma)$ there is a unique isomorphism $f_\sigma$ of polarized K3 surfaces $X \rightarrow X^\sigma$. One checks that this gives a Galois descent data for $X$ over its field of moduli (see \cite{valloni2019fields} and \cite{https://doi.org/10.48550/arxiv.2206.02560} where similar techniques are used). 
\end{proof}
This automatically holds when $L$ is not $2$-elementary (that is, $A_L$ is $2$-torsion, which happens only for finitely many lattices of a given rank) and $X$ is general in the sense that $\mu(X) = \{ \pm 1 \}$ (which happens outside a proper closed subset of the moduli space). 

Our aim now is to show that this always happens outside a proper closed subset of $\mathcal{F}_L$. 
\begin{defi} \label{defi ODelta}
    For a negative definite lattice $K$, let $\Delta(K)$ be the set of roots of $K$ and choose a decomposition $\Delta(K) = \Delta^+(K) \sqcup \Delta^-(K).$ We denote by $\OO_{\Delta}(K) \subset \OO(K)$ the isometries of $K$ which respects this decomposition.
\end{defi}
This is well-defined up to conjugation by an element of the Weyl group $W(K)$. In the next, with a triple $(T,K,L)$, we shall indicate three primitive sublattices of $\Lambda$ which are pairwise orthogonal, their orthogonal sum $T \oplus K \oplus L$ is of finite index in $\Lambda$ and $T$ has signature $(2_+,n_-)$ and $L$ is hyperbolic. 
 
\begin{defi} \label{defi rigid}
  We say that $(T,K,L)$ is \textit{rigid} if every isometry of the form $\pm \mathrm{Id}_T \oplus q \oplus \mathrm{Id}_L$ with $q \in \OO_{\Delta}(K)$ is the identity (note that this is well-defined and does not depend on the chosen decomposition of $\Delta(K)$). 
\end{defi}

\begin{rmk}
Rigidity is easy to characterize when $\mathrm{rank}(K) = 1.$ In this case $\OO(K) = \mu_2$ and therefore rigidity is automatic if $K = \langle -2 \rangle$, otherwise it is implied by the fact that none of the isometries $ \pm \mathrm{Id}_T \oplus \pm \mathrm{Id}_K \oplus \mathrm{Id}_L$ other than the identity extend to an isometry of $\Lambda$. This happens in particular whenever each $A_L,A_K,A_T$ is not $2$-torsion.
\end{rmk}
Now assume that $A_K$ is not $2$-torsion and let $g \in \OO(\Lambda)$ be an isometry which extends one of the form $\pm \mathrm{Id}_T \oplus q \oplus \mathrm{Id}_L$. Let $M$ be the saturation of $T \oplus L$ in $\Lambda$. Then $M = K^\perp$ and $K = M^\perp$ since $K$ is primitive in $\Lambda$. We assume that $\OO(K) \rightarrow \OO(A_K)$ is injective (see later about this condition). Now, if $g$ extends an isometry of the form $\mathrm{Id}_T \oplus q \oplus \mathrm{Id}_L$ we see that $q$ acts as the identity on $M$, and therefore, $q$ maps to the identity of $\OO(A_K) \cong \OO(A_M)$. Thanks to the injectivity of $\OO(K) \rightarrow \OO(A_K)$ we conclude that $g$ is itself the identity. So if $g$ is not trivial it must extend an isometry of the form $- \mathrm{Id}_T \oplus q \oplus \mathrm{Id}_L$. Using the same argument, this shows that $g^2$ is the identity. We have thus proved:
\begin{prop}
    If $(X, \iota)$ is not rigid, i.e. the triple $(T(X), \NS(X)_\iota, \iota(L))$ is not rigid, then there is a non-symplectic involution $\sigma \in \Aut(X)$ which acts trivially on $\iota(L)$. 
\end{prop}
Nikulin has classified non-symplectic involutions of K3 surfaces. He showed among other things that a K3 surface $X$ has a non-symplectic involution if and only if $\NS(X)$ embeds primitively into a certain $2$-elementary lattice $L_{r, \delta, a}$, where $r$ is the rank and $a$ is the length of the discriminant form and $\delta \in \{0,1 \}$. We do not need to explain what these lattices are, we just need to know that they are finitely many. 

Now we will comment on the injectivity of the map $\OO(K) \rightarrow \OO(A_K)$ for even definite lattices. In the reasoning above, we actually only needed the weaker statement that $\OO^*(K) \cap \OO_\Delta(K) = \{1\}$.  As explained in \cite[Rmk. 1.14.7]{MR525944}, one expects that for most definite lattices one has $\OO^*(K) \cong W(K)$. More precisely, $W(K)$ is clearly contained in $\OO^*(K)$, and one has a decomposition $\OO^*(K) = \OO^*_\Delta(X) \cdot W(K)$, where $\OO^*_\Delta(K) = \OO^*(K) \cap \OO_\Delta(K)$.
\begin{defi}
The lattice $\mathcal{L}(K) \coloneqq (K^{\OO^*_\Delta(K)})^{\perp}$ is a Leech-type lattice. 
\end{defi}
It follows that if $\OO^*_\Delta(X)$ is non-trivial, then $K$ primitively contains a Leech-type lattice. The point is that Leech-type lattices are very rare:
\begin{prop}
    The only non-trivial Leech-type lattice of rank smaller than $11$ is $E_8(2)$.
\end{prop}
Using Nikulin's results, it is easy to show that Leech-type lattices are finite when the rank is bounded. 
\begin{prop}
    There are only finitely many Leech-type lattices of rank smaller than a given $N>0$. 
\end{prop}
\begin{proof}
This is basically the same proof as in Theorem 1.14.9 in \cite{MR525944}, we give some details. Pick a number $D$ which is a multiple of $8$ and bigger than $2N$. Let $\mathcal{U}$ be the finite set of isometry classes of even, negative definite unimodular lattice of rank $D$. Then by Theorem 1.12.4 of \textit{loc. cit.} each definite lattice of rank smaller than $N$ embeds primitively into some element of  $\mathcal{U}$. In particular, every Leech-type lattice $L$ as well. Then ones uses the second point of Proposition 1.14.8 in \textit{loc. cit.} to show that there is a finite subgroup $H \subset \OO(U)$ with $U \in \mathcal{U}$ such that $L \cong (U^H)^\perp$, which proves the statement due to the finiteness of $\mathcal{U}$ and of $ \OO(U)$.
\end{proof}
\begin{prop} \label{prop univ family}
    Assume that $L$ is not $2$-elementary. Then, there is a complement of a finite union of special proper subvarieties of $\mathcal{F}_L$ whose points are all rigid, and which therefore supports a universal family. 
\end{prop}
\begin{proof}
From the discussion in the next chapter, one has that if $L \subset N$ is a primitive embedding of hyperbolic lattices, then there are only finitely many components of the Noether-Lefschetz locus parametrizing K3 surfaces with $\iota(L) \subset N \subset \NS(X)$. Now, consider all the points $(X, \iota)$ such that
    \begin{enumerate}
        \item $\mu(X) = \{ \pm 1 \}$;
        \item $A_{\NS(X)}$ is not $2$-torsion;
        \item $\NS_\iota(X)$ does not contain a Leech-type lattice;
        \item There are no lattices of the form $L_{r, \delta, a}$ sandwiched between $\iota(L)$ and $\NS(X)$.
    \end{enumerate}    
    We know that all these points correspond to rigid points by our previous discussion, so we only need to show that the corresponding locus is open. Pairs satisfying (1) and (2) are contained in proper closed subsets (due to the finiteness of $2$-elementary lattices and the fact that $L$ is not one). If $K$ contains a Leech-type lattice $\mathcal{L}$, let $\tilde{L}$ be the saturation of $\mathcal{L} \oplus L$ in $\NS(X)$. Since there are only finitely many possible $\mathcal{L}$ and $L$ is fixed, we obtain finitely many possible $\tilde{L}$. It follows that points satisfying (3) are contained in the various images $\mathcal{F}_{\tilde{L}}$ induced by all finitely many (isomorphism classes of) embeddings $L \subset \tilde{L}$, and hence they are contained in finitely many closed proper subsets. Finally, concerning the last point, since $L$ is not $2$-elementary, we see that $r > \mathrm{rank}(L)$. Thus also all the points in (4) are contained in the images of the various $L_{r, \delta, a}$ associated to primitive embeddings $L \hookrightarrow L_{r, \delta, a}$ for $r>1$. This also is a proper closed subset. 
\end{proof}
There is another notion of rigidity that is useful to us:
\begin{defi} \label{defi rigid II}
  Let $S,T$ be lattices of signature $(2_+, -)$ and let $ S \subset T$ be a primitive embedding, and let $K$ be its orthogonal complement. We say that $f$ is rigid if the only isometry of the form $q \oplus \mathrm{Id}_S$ with $q \in \OO_{\Delta}(K)$ which extends to an isometry of $T$ is the identity. 
\end{defi}
For instance, we have:
\begin{lemma}
Let $K = S^\perp$ and assume that $\OO(K) \cong \mu_2 \times W(K)$. Then $S \subset T$ is rigid whenever $T/(S \oplus K)$ is not $2$-torsion. 
\end{lemma}

\section{Maps between moduli spaces}
Consider two lattices $T, S$ of signature type $(2_+, -)$ and let $T \rightarrow S$ be a map that respects the pairing (and therefore, necessarily injective). Then we have natural maps $\mathcal{D}^{\pm}_{T} \rightarrow \mathcal{D}^{\pm}_S$ and $\OO^*(T) \hookrightarrow \OO^*(S)$ from Proposition \ref{prop funct O^*} and, therefore, we have an induced map $\mathcal{M}_T \rightarrow \mathcal{M}_{S}$ which is analytic. The fact that this is also algebraic follows from an application of a result Borel (see Theorem 3.14 \cite{zbMATH05071308}). 

At the level of Hodge structures, the map $\mathcal{K}_T^* \rightarrow \mathcal{K}^*_S$ works as follows. Let $K$ be the orthogonal complement of $T$ in $S$. Then, $S$ is an overlattice of $T \oplus K$, meaning $T \oplus K \subset S$ is of finite index. We recall that an overlattice is determined by an isotropic subgroup $H \subset A_T \oplus A_K = A_{T \oplus K}$, in the sense that $S$ is the preimage of $H$ via the natural quotient $T^\vee \oplus K^\vee \rightarrow A_T \oplus A_K$. We also have a natural isometry $A_S \xrightarrow{\sim} H^\perp/H$ (since $H$ is isotropic, the quotient on the left has a natural induced quadratic form). Given an element $(T', \gamma) \in \mathcal{K}^*(T)$ we thus obtain an isotropic subgroup $H' \subset A_{T'} \oplus A_K$ simply by putting $H' = (\gamma'\oplus \id_{A_T})(H)$. This determines an overlattice $S'$ of $T' \oplus K$ which belongs to the same genus of $S$, since they have the same signature and discriminant form. Moreover, we have isomorphisms $A_S \xrightarrow{\sim} H^\perp/H \cong (H')^\perp/H' \cong A_{S'}$ thus obtaining an element of $\mathcal{K}^*(S)$ - once we endow $S'$ with the Hodge structure induced by the one on $T$ and the trivial one on $K$.

Two maps of lattices $f, f' \colon S \rightarrow$ induce the same map $\mathcal{M}_S \rightarrow \mathcal{M}_T$ if and only if there is $g \in \OO^*(T)$ such that $f = g \circ f'$. This implies that there are only finitely many possible maps associated with some lattice $S$ which embeds primitively into $T$. Similarly, the two maps have the same image if there is $g \in \OO^*(T)$ such that $g(f(S)) = f'(S)$.
\begin{lemma} \label{lemma on induced maps}
    Assume that there is $g \in \OO^*(T)$ such that $g(f(S)) = f'(S)$, and let $f_*, f'_* \colon \mathcal{M}_{T'} \rightarrow \mathcal{M}_T$ be the two maps induced respectively by $f$ and $f'$. Then, there is $q \in \OO(A_{T'})$ such that $f_* = f_*' \circ q$, where we recall that $\OO(A_{T'})$ acts on $\mathcal{M}_{T'}$ as previously noted. 
\end{lemma}
\begin{proof}
    In fact, $f^{-1} \circ g^{-1} \circ f'$ is an isometry $\phi \in \OO(S)$ and $q$ is the unique element of $\OO(A_S)$ which acts as $\phi^{-1}$ on $A_S$.
\end{proof}
From the point of moduli spaces of K3 surfaces, we choose a primitive embedding $S \hookrightarrow \Lambda$ and we let $L$ be its orthogonal complement. We distinguish two cases. 
\subsubsection{Maps induced by primitive embeddings}
Suppose $T \rightarrow S$ is a primitive embedding. Then, we obtain another primitive embedding $T \hookrightarrow \Lambda$ given by the composition $T \rightarrow S \rightarrow \Lambda$. If we denote by $N$ the orthogonal complement of $T$ then we obtain yet another primitive embedding $f \colon L \hookrightarrow N$. Now, all of these primitive embeddings induce identifications $\mathcal{F}_L(\C) \cong \mathcal{K}^*_T$ and $\mathcal{F}_N(\C) \cong \mathcal{K}^*_S$ due to \ref{final diagram comparison}. The proof of the following fact is just a verification that we omit:
\begin{lemma}
   Under these identification, the map $\mathcal{F}_L(\C) \rightarrow \mathcal{F}_N(\C)$ corresponding to $\mathcal{K}_T^* \rightarrow \mathcal{K}^*_S$ sends $(X, \iota)$ to $(X, \iota \circ f)$. 
\end{lemma}
Using this description it is also easy to see that $\mathcal{F}_L \rightarrow \mathcal{F}_N$ is defined over $\Q$: let $\sigma \in G_\Q$ and compute: $$(f_*(X,\iota))^\sigma = (X, \iota \circ f)^\sigma = (X^\sigma, (\iota \circ f)^\sigma ) = (X^\sigma,  \sigma_* \circ \iota \circ f ) = f_*((X, \iota)^\sigma).$$
In general, all the maps induced by maps of lattices are defined over $\Q$. This is also mirrored by the fact that the corresponding Shimura datums have $\Q$ as their reflex field since the lattice is indefinite, and that the maps above are maps of Shimura varieties induced by the embedding $\SO_T \hookrightarrow \SO_S$. 
\begin{rmk}
One can write the Noether-Lefschetz locus as the union of all the images of $\mathcal{M}_S \rightarrow \mathcal{M}_T$ for $S$ varying among all the possible choices. It also suffices to make $S$ vary among the sublattices of $T$ of corank one, thus expressing the Noether-Lefschetz locus as union of divisors. 
\end{rmk}
The next proposition is a consequence of the fact that $\OO(K)$ is universally bounded for definite lattices of a given rank. This is in fact the reason why the main result holds in full generality, i.e., for any modular orthogonal variety. 
\begin{prop} \label{thm univ bound}
For any primitive embedding $T \hookrightarrow S$ and any connected component $\mathcal{M}^0_T$ of $\mathcal{M}_T$, the degree of the induced map $ \mathcal{M}^0_T \rightarrow \mathcal{M}_S$ onto its image can be bounded only in terms of $\mathrm{rank}(S) - \mathrm{rank}(T)$. 
\end{prop}

\begin{proof}
Let $\mathcal{M}^0_S$ be the connected component of $\mathcal{M}_S$ to which $\mathcal{M}^0_T$ is mapped. As explained at the beginning of this section, this can determined via lattice theory: if $\pi_0(\mathcal{M}_T)$ is determined (up to an indeterminacy of at most two) by an isomorphism class of a pair $(T', \gamma') \in \pi_0(\mathcal{K}_T)$, then the construction carried out before above determines another tuple $(S', \gamma_S') \in \pi_0(\mathcal{M}_S)$. Note that $S'$ is constructed as an overlattice of $T' \oplus K$. This means that there is a primitive embedding $T' \hookrightarrow S'$ such that the map $\mathcal{M}^0_T \rightarrow \mathcal{M}^0_S$ is induced by $\mathcal{D}_T^{\pm}/\OO^*(T') \rightarrow \mathcal{D}_S^{\pm}/\OO^*(S').$ 

Let $\OO^*(T)^+ \subset \OO^*(T)$ be the stabilizer of $\mathcal{D}_T^+$. This acts on $\mathcal{D}_T^+$, and we consider the induced map $\mathcal{D}_T^+/ \OO^*(T)^+ \rightarrow \mathcal{D}^+_S/ \OO^*(S)^+$. Let $\mathrm{Stab}^*(T) \coloneqq \{ g \in \OO^*(S) \colon g(T) = T \}$ and write $\Gamma \subset \OO(T)$ for the image of $\mathrm{Stab}^*(T)$ in $\OO(T)$ under the natural restriction map. We have $\OO^*(T) \subset \Gamma$ due to Lemma \ref{lemma stable}. We put $\mathrm{Stab}^*(T)^+ \coloneqq \mathrm{Stab}^*(T,S) \cap \OO^*(S)^+$ and similarly $\Gamma^+ \subset \OO(T)$ for the image of $\mathrm{Stab}^*(T)^+$ in $\OO(T)$. The natural injection $\OO^*(T) \hookrightarrow \OO^*(S)$ respects the induced decomposition $\OO^*(T) \cap \OO_T(\R)^{ \pm , \pm }$. This implies that $\OO^*(T)^+ \subset \mathrm{Stab}^*(T)^+$ and hence that $\OO^*(T)^+ \subset \Gamma^+$ as a normal subgroups. The degree of $\mathcal{D}_T^+/ \OO^*(T)^+ \rightarrow \mathcal{D}^+_S/ \OO^*(S)^+$ is then given by $2^{-1} [ \Gamma^+ \colon \OO^*(T)^+$ or $[ \Gamma^+ \colon \OO^*(T)^+]$ depending on whether $-\mathrm{Id} \in \Gamma  \setminus \OO^*(T)$ or not. Let now $K = T^\perp$ be the orthogonal complement of $T$ in $S$. It follows that $K$ is a definite lattice and therefore that $\OO(K)$ is finite. Moreover, the cardinality of $|\OO(K)|$ can be bounded only in terms of $\mathrm{rank}(K)$, due to the fact that every finite group acting on some $\Z^n$ embeds to $\GL(\FF_{3}^n)$ under the natural quotient map. One has an injection 
$$\mathrm{Stab}^*(T) \hookrightarrow \OO(T) \times \OO(K),$$
and we consider the induced map 
$$\mathrm{Stab}^*(T) \rightarrow \OO(K).$$
An element $g \in \mathrm{Stab}^*(T)$ that maps to the identity of $\OO(K)$ is an element of $\OO^*(S)$ that stabilizes $T$ and that acts trivially on $K$. By embedding $S$ into an even unimodular lattice $U$, we know that $g$ extends to an isometry of $U$ by requiring it to be the identity on the orthogonal complement of $S$. It follows that this isometry preserves $T$ and acts as the identity also on the orthogonal complement of $T$, and using Lemma \ref{lemma stable} this shows that $\ker(\mathrm{Stab}^*(T) \rightarrow \OO(K)) = \OO^*(T)$ and therefore that $[\Gamma \colon \OO^*(T)] $ is universally bounded depending only on $\mathrm{rank}(K)$.
\end{proof}
We also remark that the cardinalities of the preimages of Picard generic points behave as one expects:
\begin{lemma} \label{lemma same preimages}
Let $(T', \gamma) \in \mathcal{K}^*_T$ whose underlying Hodge structure is $\Q$-irreducible (meaning that there are no $(0,0)$-classes) and such that $\Aut_{\Hdg}(T') = \{ \pm 1 \}$. If $f \colon \mathcal{K}^*_T \rightarrow \mathcal{K}^*_S$ induced by a primitive embedding $T \hookrightarrow S$ then $|f^{-1}(f((T', \gamma)))|$ is independent on $(T', \gamma)$.
\end{lemma}
\begin{proof}
    In fact, in this case, the preimages can be computed in purely lattice theoretical terms, independently from the Hodge structure on $T'$. 
\end{proof}
\subsubsection{Maps induced by finite index inclusions} \label{subsection maps induced}
Let $T' \subset T$ be a finite index inclusion. In this case, $T$ is an overlattice of $T'$, which is determined by an isotropic subgroup $H \subset A_{T'}$. Clearly $\mathcal{D}_{T'}^{\pm} = \mathcal{D}_{T}^{\pm} $, and therefore the induced map $\mathcal{M}_{T'} \rightarrow \mathcal{M}_T$ is the quotient map induced by the inclusion $\OO^*(T') \subset \OO^*(T)$. 

One can interpret these maps as follows. Proposition 6.6. in \cite{MR893604} and its proof show that when $T/T'$ is cyclic, the corresponding maps are induced by associating to $(X,\iota)$ some moduli space of stable sheaves on $X$, with a Mukai vector whose $(1,1)$-part can be chosen in $L$ (we refer the reader to the proof for an explicit construction). In general, one can always filter $T' \subset T$ via $T'_i \subset T'_{i+1}$ with $T'_0 = T', T'_n = T$ and $T'_{i+1}/T'_{i}$ is cyclic. This gives a corresponding factorization of the map $\mathcal{M}_{T'} \rightarrow \mathcal{M}_{T}$ all of whose factors can be interpreted as taking $X$ to some moduli spaces of sheaves on $X$. This also shows that the maps are naturally defined over $\Q$. 

Now, we construct the maps from the introduction, which can be considered the building blocks for every other map of this type. Fix a prime $p$ coprime to $\mathrm{discr}(T)$. We begin by classifying sublattices $T' \subset T$ such that $T/ T' \cong \Z/ p$ (see also \cite{MR3616011}). Due to the coprimality of $p$ and $\mathrm{discr}(T)$, we have a natural splitting $A_{T'} \cong A_{T} \oplus A_p$ where $A_p$ is the $p$-part of $A_{T'}$, and moreover $A_p \cong \Z/p \Z \oplus \Z / p \Z$ or $A_p \cong \Z/ p^2 \Z$ necessarily. The first case is part of Kneser's theory of neighboring lattices, and we call it the isotropic case (see \cite{https://doi.org/10.48550/arxiv.2206.02560} for an application to K3 surfaces). In the second case, let $\epsilon$ be a generator of $A_p$ and consider the quadratic form $q \colon A_p \rightarrow \Q_p / \Z_p$. Then $q(\epsilon) = \theta/p^2$ for some $\theta \in \Z_p^{\times}$ (as we shall soon show). Let
\begin{enumerate}
    \item $\mathcal{L}^+_T(p)$, the set of sublattices such that $\theta$ is a square;
    \item $\mathcal{L}^{-}_T(p)$, the set of sublattices such that $\theta$ is not a square;
    \item $\mathcal{L}^0_T(p)$, the set of sublattices corresponding to the isotropic case.
\end{enumerate}
The lattices can be characterized as follows. Consider $T \otimes \mathbf{F}_p$ as a $\mathbf{F}_p$-vector space. Then, the pairing of $T$ induces a perfect pairing on $T \otimes \mathbf{F}_p$ due to the coprimality of $p$ and $\mathrm{discr}(T)$. A subgroup $T' \subset T$ of index $p$ corresponds to a hyperplane $H \subset T \otimes \mathbf{F}_p$, and therefore $\ell \coloneqq H^{\perp}$ is a line in $T \otimes \mathbf{F}_p$ by dimension reason. Let $\epsilon$ be a generator of $\ell$, then
\begin{lemma} In the same notation as before
    \begin{enumerate}
    \item $L' \in \mathcal{L}^{0}_T(p)$ if and only if $(\epsilon, \epsilon) = 0$;
    \item $L' \in \mathcal{L}^{+}_T(p)$ if and only if $(\epsilon, \epsilon) \in \mathbf{F}_p^{\times 2}$;
    \item $L' \in \mathcal{L}^{-}_T(p)$ if and only if $(\epsilon, \epsilon) \in \mathbf{F}_p^{\times} \setminus \mathbf{F}_p^{\times 2}
    $.
\end{enumerate}

\end{lemma}
\begin{proof}
    Assume that $\ell$ is isotropic and let $v \in T$ be a primitive vector that generates it. The next arguments are the same found in Kneser's theory of neighboring lattices. One first shows that one can choose such $v$ with $v^2 \in 2p^2 \Z$. Then, $T' \coloneqq \{ w \in T \colon (w,v) \in p \Z \}$ is the sublattice of index $p$ which corresponds to $\ell^\perp$ From this description, it is also clear that the dual of $T'$ contains the vector $v/p$. Moreover, this vector is isotropic, meaning that it generates a subgroup $H \cong \Z/ p \Z \subset A_p$ which is isotropic, and the overlattice it generates is different from $T$ (it is called a $p$-neighbor of $T$). Thus we have $\Z/p \Z \oplus \Z/ p \Z \cong A_p.$ On the other hand, suppose that $A_p$ is isomorphic to $\Z/p\Z \oplus \Z/ p \Z$ as groups. We know that there is an isotropic subgroup $H \subset A_p$ and we know that the bilinear pairing $A_p \times A_p \rightarrow \Q/ \Z$ is perfect. Thus there must be another unique isotropic subgroup $H'$ such that $H' \cap H = 0$ and therefore $A_p = H \oplus H'$. Then $H'$ corresponds to an everlattice $T''$ of $T'$ of index $p$ (the $p$-neighbor of $T$) and $T'' \cap T = T'$. Consider now $\{ v \in T \colon (v,T') \subset p \Z \}$. This is generated by $v$ and $pT$, but since $v \in T'$, we conclude that $\ell \subset \ell^\perp$. 

    For the second and third case, assume that $A_p \cong \Z/ p^2 \Z$ and let $\epsilon$ be a generator of $A_p$, and put $q(\epsilon) = a \in \Q_p/ \Z_p$. We know that $pA_p$ must be an isotropic subgroup, which implies that $a \in p^{-2} \Z_p$. Assume that $a \in p^{-1} \Z_p$. Then $pA_p$ is not only isotropic, but it pairs to zero with the whole $A_p$, which is impossible since the bilinear pairing on $A_p$ is perfect. So we write $a = \theta /p^2$ for some $\theta \in \Z$ and we pick a vector $v \in (T')^\vee$ such that $(T')^\vee = T' + \Z v$. We can write $v = w/p^2$ for some $w \in T'$, which also satisfies $w/p \in T$. One computes $(v,v) = (w,w)/p^4$. Therefore $(w,w)$ is divisible exactly by $p^2$. This means that $w/p \in T$ is not isotropic mod $p$. Moreover, $T' = \{ v \in T \colon (v,w/p) \in p \Z \}$. Thus the image of $T'$ in $T \otimes \mathbf{F}_p$ corresponds to the orthogonal complement of the line generated by $w/p$, proving both points (2) and (3). 
\end{proof}
We thus identify the various sets $\mathcal{L}^\bullet_T(p)$ with the corresponding sets of lines in $T \otimes \mathbf{F}_p$. 
\begin{prop} \label{prop transitivity action}
Assume that $\mathrm{rank}(T) \geq 5$. Then $\OO^*(T)$ acts transitively on $\mathcal{L}^\bullet_T(p)$ for any odd prime $p$ coprime to $\mathrm{discr}(T)$. 
\end{prop}
\begin{proof}
The proposition amounts to prove that $\OO^*(T)$ acts transitively on the corresponding sets of lines in $T \otimes \mathbf{F}_p$. Let $\mathrm{Spin}_T$ be the spin group associated to $T$, whose generic fiber is simply connected and semisimple, and it is the universal cover of $\SO_{T_\Q}$. Strong approximation \cite[Thm.7.12]{zbMATH00052370} asserts that $\mathrm{Spin}_T(\Q)$ is dense in $\mathrm{Spin}_T(\A_f)$ whenever $T$ is not definite, which implies that $K \cap \mathrm{Spin}_T(\Q)$ is dense in $K$ for any compact-open subgroup $K \subset \mathrm{Spin}_T(\A_f)$. By choosing $K = \mathrm{Spin}_T(\widehat{\Z})$ we see that $\mathrm{Spin}_T(\Z)$ is dense in $\mathrm{Spin}_T(\widehat{\Z})$. Since for $p$ odd and coprime to $\mathrm{discr}(T)$ both the groups $\SO_T$ and $\mathrm{Spin}_T$ have good reduction at $p$, we obtain a natural isomorphism $\SO_T \otimes \FF_p \cong \SO_{T \otimes \FF_p}$ and similarly for $\mathrm{Spin}_T$. A consequence of strong approximation is that $\mathrm{Spin}_T(\Z) \rightarrow \mathrm{Spin}_T(\FF_p)$ is surjective. Over $\FF_p$ we have the spin norm $\mathrm{spn} \colon \OO_T(\FF_p) \rightarrow \mu_2$ whose kernel intersected with $\SO_T(\FF_p)$ coincides with the image of $\mathrm{Spin}_T(\FF_p)$. We recall that if $r \in \OO_T(\FF_p) $ is a reflection associated to a non-isotropic vector $v$, then $\mathrm{spn}(r) = (v,v) \in \FF_p^\times/ \FF_p^{\times 2} = \mu_2$. 

Now, we claim that every element of $\mathcal{L}^{\bullet}_T(p)$ is fixed by an element of $\SO(T \otimes \FF_p)$ of arbitrary spinor norm (this is why we need $\mathrm{rank}(T) \geq 5$). If $\ell$ is isotropic, then $\ell^\perp/ \ell$ is a non-degenerate quadratic space of dimension at least three, hence it represents all the values in $\FF_p$. This implies that we can find $v \in \ell^\perp$ non-isotropic and with $(v,v)$ arbitrary. But then the reflection at $v$ clearly fixes $\ell$ and has arbitrary Spinor norm. Then we can choose an even product of reflections to obtain the desired element. If $\ell$ is non-isotropic, then $\ell^\perp$ is a non-degenerate quadratic space of dimension at least $4$, and we do as before. 

We now show that the image of $\mathrm{Spin}_T(\FF_p)$ acts transitively on each set of lines. Let $\ell, \ell' \in \mathcal{L}_T^\bullet(p)$ and let $g \in \OO(T \otimes \FF_p) $ be an element which sends $\ell$ to $\ell'$. If $\det(g) = 1$ and $\mathrm{spn}(g) =1$, there is nothing to prove. If $\det(g) = 1$ and $\mathrm{spn}(g) = -1$ we choose $r \in \OO(T \otimes \FF_p)$ to be a product of two reflections of different spinor norm which fix $\ell$, so that $r \in \SO(T \otimes \FF_p)$ and $r$ has Spinor norm $-1$. Then also $gr$ sends $\ell$ to $\ell'$ and it belongs to the image of $\SO_T(\FF_p)$. If $\det(g) = -1$ and $\mathrm{spn}(g) = \pm 1$ then we reason analogously. Finally, if $d \in \Z$ is the discriminant of $T$ then one has a natural injection $A_T \subset T/dT$, which is induced by $dT^\vee \subset T$. Due to strong approximation the map $\mathrm{Spin}_T \rightarrow \mathrm{Spin}_T( \mathbf{Z}/ d \mathbf{Z}) \times  \mathrm{Spin}_T( \mathbf{F}_p)$ is surjective. This implies that $\mathrm{Spin}_T^*(\Z)$ (the elements which map to $\SO^*(T)$) surjects onto $\mathrm{Spin}(T \otimes \mathbf{F}_p)$. Finally, this implies that the image of $\OO^*(T)$ in $\OO(T \otimes \FF_p)$ contains the image of $\mathrm{Spin}(T \otimes \FF_p)$, and the proposition is proved. 
\end{proof}
From this and Lemma \ref{lemma on induced maps} it follows that any lattice $T$ and any prime $p$ big enough, there are associated three spaces $\mathcal{M}_T(p)^\bullet$ together with three maps $\pi_p^\bullet \colon \mathcal{M}_T(p)^\bullet \rightarrow \mathcal{M}_T$ which are well-defined up to a $\Q$-automorphisms of $\mathcal{M}_T(p)^\bullet$. The degrees of the maps are greater than the respective $| \mathcal{L}_T(p)^{\bullet} |$. 

We now give enough conditions on $p$ and $T$ under which the map $\OO^*(T) \rightarrow \OO(T \otimes \FF_p)$ is surjective, which will be useful in the next chapter. 
\begin{cor} \label{coro surj OO mod p}
    Assume that $\mathrm{rank}(T) \geq 5$ and that $\mathrm{rank}(T) - 2 \geq \ell(A_T)$. Let $p \equiv 3 \mod 4$ be a prime which is coprime to $\mathrm{discr}(T)$. Then the natural map $\OO(T) \rightarrow \OO(A_T) \times \OO(T \otimes \FF_p)$ is surjective. 
\end{cor}
\begin{proof}
Since $\mathrm{rank}(T) - 2 \geq \ell(A_T)$ the lattice $T$ contains primitive classes $\delta, \delta'$ of square $-2$ and $2$ respectively, as an application of Theorem 1.12.2 of \cite{MR525944}. We know by Proposition \ref{prop transitivity action} that the image of $\OO^*(T)$ mod $p$ contains the image of $\mathrm{Spin}_T(\FF_p)$, but since it also contains the reflections $r,r'$ associated to $\delta$ and $\delta'$ respectively, we conclude that the image is the whole $\OO_T(\FF_p)$ (because $-1$ is not a square mod $p$ and hence $\mathrm{spn}(r) \neq \mathrm{spn}(r')$). The rest follows from the surjectivity of $\OO(T) \rightarrow \OO(A_T)$ which is automatic when $\mathrm{rank}(T) - 2 \geq \ell(A_T)$. 
\end{proof}
\begin{rmks} \label{remark field of defi moduli spaces}
\begin{enumerate}
    \item Let $T$ be a lattice of signature $(2_+, n_-)$ which embeds primitively into $\Lambda$. Then it is not true that any sublattice $S \subset T$ of finite index embeds primitively into the K3 lattice, and in fact, this is false as soon as $\mathrm{rank}(T) \geq 12$ (but it is always true when $\mathrm{rank}(T) \leq 10$). For simplicity, assume that the index of $S \subset T$ is coprime to $\mathrm{discr}(T)$, which implies that we have a natural decomposition $A_{S}  = A_T \perp H_{S}$ and $\OO(A_{T'}) = \OO(A_T) \times \OO(H_S)$. Note that $H_{S}$ is an isotropic subgroup that corresponds to the overlattice $T$. Now, for any pair $(S', \gamma') \in \mathcal{K}^*_{S}$ the isotropic subgroup $H_{S}$ can be considered as an isotropic subgroup of $A_{S'}$ using the map $\gamma'$. Thus, $S'$ can be naturally considered as a sublattice of some element of $\mathcal{K}^*_T$. Under our primality condition, this means that $\mathcal{K}^*_{T'}$ parametrizes triples $((T',\gamma'), S', \gamma'')$ where $(T', \gamma') \in \mathcal{K}^*_T$ and $S' \subset T'$ is a finite-index sublattice of index coprime to $\mathrm{discr}(T)$, and $\gamma''$ is an isomorphism $H_{S} \cong H_{S'}$, so that $(S', \gamma' \oplus \gamma'') \in \mathcal{K}_S$. Using this description, one can define a natural $G_\Q$ action on $\mathcal{M}_{T'}$ whenever $T'$ has an overlattice which embeds primitively into the K3 lattice, and then show that $\mathcal{M}_{T'}$ has a natural model over $\Q$ as in Theorem \ref{thm defined Q}. This is again mirrored by the fact that for any lattice $T$ with $n_- > 0$ the corresponding Shimura varieties have reflex field $\Q$.
    \item Let $S \subset T$ be a finite index inclusion and assume that $T$ embeds into the K3 lattice with orthogonal complement $L$. If $(X, \iota)$ is a $L$-polarized K3 surface defined over a number field $K$ and $\pi \colon \mathcal{M}_S \rightarrow \mathcal{M}_T$ is the map induced by $S \subset T$, then a $K$-rational preimage of $\pi$ of the moduli point corresponding to $(X, \iota)$ yields a sublattice $S' \subset T(X_\C)$ in the same genus of $S$, and such that $S \otimes \widehat{\Z} \subset T(X_\C) \otimes \widehat{\Z}$ is invariant under $G_K$ and $G_K$ acts trivially on $A_S$. On the other hand, each such sublattice yields a $K$-rational preimage by construction, well-defined up to the action of $\OO(A_{T'})$.
\end{enumerate}
\end{rmks}
\section{Proofs} \label{section proof}
The proof below works for any lattice $T$, without any further assumption. We remark that if $\mathrm{rank}(T) \geq 5$ and $\mathrm{rank}(T) - 3 \geq \ell(A_T)$, then every $T' 
\in \mathcal{L}^{\bullet}_T(p)$ satisfy automatically $\mathrm{rank}(T') - 3 \geq \ell(A_{T'})$, so that the spaces $\mathcal{M}_T^\bullet(p)$ are geometrically irreducible.
\begin{proof}[Proof of Theorem \ref{main thm}]
    Consider a prime $p$ coprime to $\mathrm{discr}(T)$. Pick a rational point $x \in \mathcal{M}_T(K)$ and assume that it belongs to the Noether-Lefschetz locus. Then, we can find a primitive embedding $S \hookrightarrow T$ such that $x$ is the image of a Picard generic point of $\mathcal{M}_S$ under the induced map on moduli spaces. Pick any lattice $T' \subset T$ belonging to some $\mathcal{L}_T^\bullet(p)$ such that $S \cap T' = S$. Then also $S \hookrightarrow T'$ is a primitive embedding, and we obtain a commutative diagram of morphisms of $\Q$-varieties
     \begin{equation} \label{triangle}
\begin{tikzcd}[row sep=tiny]
                         & \mathcal{M}_{T}^{\bullet}(p) \arrow[dd, "\pi^{\bullet}_p"]             \\
\mathcal{M}_S \arrow[ur,] \arrow[dr,] & \\ & \mathcal{M}_T
\end{tikzcd}
\end{equation}
Now, we can assume that $x$ is rigid in the sense of Definition \ref{defi rigid}. This implies that $(X, \iota)$ has a model over $K$. If we look at the Galois action on $\NS(\overline{X})$, then this is the identity on $\iota(L)$, and on $\NS_\iota(\overline{X})$ it must factorize through the inclusion $\OO_\Delta(\NS_\iota(\overline{X})) \subset \OO(\NS_\iota(\overline{X}))$, where we choose the decomposition of $\Delta(X)$ induced by the effective classes. Thus, by passing to the extension $K'$ that kills the Galois action, we obtain a rational point in $\mathcal{M}_S$, whose image in $\mathcal{M}_T^{\bullet}(p)$ gives us the rational preimage we were after. Moreover, if $L \subset \NS(X)$ is also rigid in the sense of Definition \ref{defi rigid II}, then we must have that $G_K$ acts as the identity on $\NS_{\iota}(X)$, so we can choose $K' = K$. 

To prove point (2), we assume that $\mathrm{rank}(T) \geq 4$. Assume that $x$ is Picard generic and consider the Galois representation $\rho_x \colon G_K \rightarrow \OO_T(\widehat{\Z})$ associated with $X$ (where we identify $T$ with $T(X_\C)$). Then the main result in \cite{cadoret_moonen} shows that the image of $\rho_x$ is a compact-open subgroup for the profinite topology and, in particular, that there is an integer $N$ such that for each $p > N$ the $p$-component $\rho_{x,p} \colon G_K \rightarrow \OO_T(\Z_p)$ is surjective. But then $G_K$ acts transitively on the various sets $\mathcal{L}_T^\bullet(p)$, showing that if $x'$ is a $K'$-rational preimage of $x$ under $\pi_p^\bullet$, then $[K' \colon K] \geq | \mathcal{L}_T^\bullet(p)|$. Since $| \mathcal{L}_T^\bullet(p)|$ is unbounded as $p$ varies, this concludes the proof. 
\end{proof} 
Now we explicitly describe the unique rational preimage which we obtain when $\rho(X) = \mathrm{rank}(L) + 1$. Recall that if $T' \subset T$ is an overlattice, then $T/T' \subset A_{T'}$ is an isotropic subgroup and, moreover, $\Hom(T/T', \Q/\Z) \cong (T')^\vee/T^\vee$ is a quotient of $A_{T'}$. Next, for any lattice $T$, we write $T_p$ for $T \otimes \Z_p$ considered as a $p$-adic lattice. Note that $A_{T_p} \cong A_T \otimes \Z_p $ naturally. Assume that $(X, \iota)$ is defined over a number field $K$ and that $\rho(\bar{X}) = \mathrm{rank}(L) + 1$. Then there is a unique sublattice $T' \in \mathcal{L}_T^\bullet(p)$ that contains $T(X)$. We assume that we have already enlarged $K$ so that $G_K$ acts trivially on $\NS(\bar{X})$, and due to the discussion after Remark \ref{remark field of defi moduli spaces} we only need to show that $G_K$ acts trivially on $A_{T'}$ under the induced action (and in particular that stabilizes $T'$). In what follows, we are tacitly working after having base-changed $X$ to $\C$, bearing in mind that all the groups involved have a natural $G_K$-action as well. 
\begin{itemize}
    \item Suppose that both $\NS_\iota(X)$ and $T(X)$ have discriminants coprime to $p$. Then $T(X) \otimes \mathbf{F}_p \subset T \otimes \mathbf{F}_p $ is non-degenerate and $\NS_\iota(X) \otimes \mathbf{F}_p  \subset \otimes \mathbf{F}_p $ is a line $\ell$ which is its orthogonal complement. We let $T'$ be the unique sublattice of index $p$ of $T$ such that $T' \cap \NS_\iota(X) = p \NS_\iota(X)$. This is also given by $\{ w \in T \colon (w, \NS_\iota(X)) \in p\Z \}$. But then $A_{T'} = A_T \oplus A_{pK} \cong A_T \oplus \NS_\iota(X)/p^2/p^2$. Now, $G_K$ acts trivially on $A_T$ since $(X, \iota)$ is defined over $K$ and $G_K$ acts trivially on $\NS_\iota(\overline{X}) /p^2 \NS_\iota(\overline{X}) $ since it acts trivially on $\NS(\overline{X})$, so this clearly gives us a rational preimage under $\pi_p^\pm$. Note that all other lines of $T \otimes \mathbf{F}_p$ can be generated by $(v,w)$ with $v \in T(X) \otimes \mathbf{F}_p$ and $w \in \NS_\iota(X) \otimes \mathbf{F}_p$. Since $\Br(\overline{X})[p]  \cong T(X) \otimes \mathbf{F}_p$ in this case, one can deduce that only the sublattice above is fixed a priori by $G_{K}$. 
\item If instead $T(X)$ has discriminant divisible by $p$, consider again in $T \otimes \mathbf{F}_p$ the hyperplane $T(X) \otimes \FF_p$ and the line $\ell = \NS_\iota(X) \otimes \FF_p$. Then $\ell$ is isotropic and orthogonal to $T(X) \otimes \mathbf{F}_p$, so that $\ell = (T(X)\otimes \mathbf{F}_p)^{\perp}$ necessarily. As before, we put $T' = \{ w \in T \colon (w, \NS_\iota(X)) \subset p \Z \}$. This contains both $\NS_\iota(X)$ and $T(X)$ as primitive sublattices, giving us a chain of inclusions $T(X) \oplus \NS_\iota(X) \subset T' \subset T$. Thus, $T'$ is given by an isotropic subgroup $H' \subset A_{T(X)} \oplus A_{\NS_\iota(X)}$. Now, $G_{K}$ acts trivially on $A_{T(X)} \oplus A_{\NS_\iota(X)}$, and hence also on its subquotient $A_{T'} \cong H'^{\perp}/H'$.
\end{itemize}
\begin{rmk} 
In fact, the number of rational preimages is related to the Picard number of $(X, \iota)$. For simplicity, assume that $\mathrm{discr}(T(X))$ is not divisible by $p$. Let $T$ be the saturation of $T(X) \oplus \NS_{\iota}(X)$. Then the image of $T(X)$ in $T \otimes \FF_p$ is non-degenerate with orthogonal complement $\NS_{\iota} \otimes \FF_p$. If $\NS(\overline{X})$ is defined over $K$, each line in $\NS_{\iota}(X) \otimes \FF_p $ gives a rational preimage under some $\pi_p^{\bullet}$. But non-degenerate quadratic spaces can be only of two types over finite fields of odd characteristic, and one has formulas for the number of isotropic/non-isotropic lines. Thus, one can in principle guess the Picard number simply by counting the rational preimages of $\pi_p^{\bullet}$.
\end{rmk}
\begin{proof}[Proof of Corollary \ref{cor implication}]
Choose a cover $\tilde{\mathcal{M}}_T \rightarrow \mathcal{M}_T$ of degree $d$ that supports a family of $L$-quasipolarized K3 surfaces and let $M = \mathrm{max}\{ | \OO_\Delta(\NS_\iota(X))| \}$ for all possible $(X, \iota) \in \mathcal{F}_L(\C)$. Define now $C = C(L,N)$ as $\mathrm{max} \{ |\Br(\overline{X})^{G_K}|\}$ where $(X,\iota)$ varies over all K3 surfaces defined over any extension $K/ \Q$ of degree $[K \colon \Q] \leq N \cdot M \cdot d$ and such that $\iota \colon L \xrightarrow{\sim} \NS(X_\C)$. This maximum exists because we are assuming V\'{a}rilly-Alvarado conjecture. 

Let now $K/\Q$ be any extension of degree $[K \colon \Q] \leq N$ and choose $p > C$. Let $x \in \mathcal{M}_T(K)$ be Picard generic. Then there is an extension $K'/K$ of degree $[K' \colon K] \leq d$ and a $L$-quasipolarized K3 surface $(X, \iota)/K'$ representing $x$. Let $K''/K'$ be an extension of degree $[K'' \colon K'] \leq M$, and let $y$ be a $K''$-rational point in some fiber of one of the maps $\pi_p^\bullet$. This corresponds to an index $p$ subgroup $T' \subset T(X_\C)$ such that $T' \otimes \widehat{\Z}$ is invariant by $G_K$ and such that the induced action on $A_{T'}$ is trivial. But any such subgroup is induced by a unique cyclic subgroup $C \subset \Br(\overline{X})[p]$. Moreover, $C \cong (T')^\vee/T^\vee$ naturally, which implies that $C$ is a quotient of $A_{T'}$. Since $G_{K''}$ acts trivially on $A_{T'}$, we deduce that $p$ divides $|\Br(\overline{X})^{G_{K''}}|$. But this is impossible since $[K'' \colon K] \leq N \cdot M \cdot d$, which contradicts the fact that $x$ is Picard generic. 
\end{proof}

\subsection{Consequences of Bombieri-Lang} We begin by proving point (1) of Theorem \ref{cor shafa}. 
\begin{proof}
By \cite[Thm.1.3]{MR3802299}, we know that in this case the (components of the) space $\mathcal{M}_T(p)^{\bullet}$ are of general type for almost all $p$. This implies that $\mathcal{M}_T(p)^\bullet(K)$ is not Zariski dense for every number field $K$ under the Bombieri-Lang conjecture, and then we conclude by point (2) of Theorem \ref{main thm}.
\end{proof}
To prove point (2), we slightly generalize the ideas behind Theorem \ref{main thm} and use the following result:
\begin{thm}(\cite[Thm.1.1.]{MR3802299}) \label{thm ma}
Up to scaling, there are only finitely many lattices $T$ of signature $(2,17)$ such that $\mathcal{D}^{\pm}_T/ \OO(T)$ is not of general type.
\end{thm}
Note that the quotient here is by the full orthogonal group. Let $p$ be a prime number as in Corollary \ref{coro surj OO mod p} and consider the surjective map $\OO(T) \rightarrow \OO(A_T) \times \OO_T(\FF_p)$. For any sublattice $pT \subset T' \subset T$, let $\mathrm{Stab}^*(T')$ be its stabilizer in $\OO^*(T)$ and $\mathrm{Stab}^*(T' \otimes \widehat{\Z}) \subset \OO^*_T(\widehat{\Z})$ be the stabilizer of $T' \otimes \widehat{\Z}$ in $T \otimes \widehat{\Z}$. Then we have natural inclusions 
\begin{equation} \label{eq incl}
\OO^*(T') \subset \mathrm{Stab}^*(T') \subset \OO^*(T) \text{ , } \OO^*(T') \subset \mathrm{Stab}^*(T') \subset \OO(T')\end{equation}
and similarly for their adelic counterparts. Now, consider the varieties $$\mathcal{M}_{T/T'} \coloneqq \OO_{T'}(\Q) \backslash \mathcal{D}_{T'}^{\pm} \times \OO_{T'}(\A_f ) / \mathrm{Stab}^*(T' \otimes \widehat{\Z})$$
and 
$$\tilde{\mathcal{M}}_{T'} \coloneqq \OO_{T'}(\Q) \backslash \mathcal{D}_{T'}^{\pm} \times \OO_{T'}(\A_f) / \OO_{T'}(\widehat{\Z}).$$
Then we have the coverings $\mathcal{M}_{T'} \rightarrow \mathcal{M}_{T/T'} \rightarrow \mathcal{M}_T$ and $\mathcal{M}_{T'} \rightarrow \mathcal{M}_{T/T'} \rightarrow \tilde{\mathcal{M}}_{T'}$ induced by \ref{eq incl}.
Note that both $\mathcal{M}_{T/T'}$ and $\tilde{\mathcal{M}}_{T'}$ are quotients of $\mathcal{M}_{T'}$ by subgroups of $\OO(A_{T'})$. 
\begin{cor}
    If $p$ is big enough and $pT \subsetneq T' \subsetneq T$ then the components of $\tilde{\mathcal{M}}_{T'}$ are all of general type.
\end{cor}
\begin{proof}
In fact, each connected component dominates a variety of the form $\mathcal{D}^+_{T''}/ \OO(T'')^+$ for some lattice $T''$ in the same genus as $T'$. But then we use Theorem \ref{thm ma} and choose $p$ big enough so that, for any lattice $S$ whose quadratic form is primitive (i.e., the value group of $S$ is $\Z$) and of rank $19$ with $p \leq \mathrm{discr}(T)$, we have that $\mathcal{D}_{T}/ \OO(T)^+$ is of general type. Then since $p$ clearly divides the discriminant of $T'$ whenver $T' \neq T$, we conclude. 
\end{proof}
Now we prove point (2) of Corollary \ref{cor shafa}. Let $x \in U(K)$ be a rational point of $\mathcal{M}_{T}$ which belongs to the Noether-Lefschetz locus, and let $(X, \iota)/K$ be the corresponding K3 surface. We can view $T(X_\C) \subset T$ as a proper primitive sublattice. But then the Galois representation associated with $X$ must stabilize $T(X_\C) \otimes \widehat{\Z}$, and therefore it also stabilizes $T(X_\C) \otimes \FF_p \subset T \otimes \FF_p$. But if we let $pT \subset T' \subset T$ be the sublattice corresponding to the subspace $T(X_\C) \otimes \FF_p$, then $pT \subsetneq T' \subsetneq T$ and $G_K$ must also stabilize $T' \otimes \hZ \subset T \otimes \hZ$. This shows that there is a $K$-rational point $x' \in \mathcal{M}_{T/T'}(K)$ which maps to $x$ via the covering $\mathcal{M}_{T/T'} \rightarrow \mathcal{M}_{T}.$ Since each component of $\mathcal{M}_{T/T'}$ is of general type and since there are only finitely many sublattices between $pT$ and $T$, we conclude the proof thanks to the Bombieri-Lang conjecture. 

\subsection{Final remarks}
The following conjecture of Shafarevich implies that $\mathcal{F}_L(K)$ is never Zariski-dense:
\begin{conj} \label{sha conj}
    Let $N>0$ be an integer. Then, there are only finitely many lattices that can appear as $\NS(\overline{X})$ where $X/K$ is a K3 surface and $[K \colon \Q] \leq N$. 
\end{conj}
To attempt to deduce this from the Bombieri-Lang conjecture and the ideas above, one has first to answer the following highly non-trivial question:
\begin{question} \label{question mix}
Is it true that the positive dimensional components of the Zariski closure of $\mathcal{F}_L^{nl}(K)$ are contained in $\mathcal{F}_{L}^{nl}$?
\end{question}
Finally, there is an analogue of Conjecture \ref{sha conj} for abelian varieties proposed by Coleman (unpublished) which is known to imply many uniform bounds for the arithmetic of abelian varieties and K3 surfaces (see \cite{MR3815154} and \cite{wrap131525}). 

\bibliography{Bibliografia}
\bibliographystyle{alpha}

 \end{document}